\documentclass[a4paper,13pt]{amsart}
\usepackage{amsmath, amsfonts,amsthm,times,graphics}
\usepackage{mathrsfs}
\usepackage{amsmath, amssymb, amsfonts}
\usepackage{amssymb}
\usepackage{ifthen}
\usepackage[usenames]{color}
\usepackage[latin1]{inputenc}
\usepackage{diagbox}
\usepackage{indentfirst} % Ê×ÐÐ×Ô¶¯Ëõ½ø
\usepackage{subfigure}
\usepackage{float} % \figure{H}
\usepackage[active]{srcltx}
 \makeatletter
\renewcommand*\subjclass[2][2000]{%
  \def\@subjclass{#2}%
  \@ifundefined{subjclassname@#1}{%
    \ClassWarning{\@classname}{Unknown edition (#1) of Mathematics
      Subject Classification; using '1991'.}%
  }{%
    \@xp\let\@xp\subjclassname\csname subjclassname@#1\endcsname
  }%
}
 \makeatother

\newtheorem*{ThmA}{Theorem A}
\newtheorem*{ThmB}{Theorem B}
\newtheorem*{ThmC}{Theorem C}
\newtheorem*{ThmD}{Theorem D}
\newtheorem*{ThmE}{Theorem E}
\newtheorem*{ThmF}{Theorem F}
\newtheorem*{ThmG}{Theorem G}
\newtheorem*{ThmH}{Theorem H}
\newtheorem*{ThmI}{Theorem I}
\newtheorem*{ThmJ}{Theorem J}
\newtheorem*{ThmK}{Theorem K}
\newtheorem*{ThmL}{Theorem L}
\newtheorem*{ThmM}{Theorem M}
\newtheorem*{ThmN}{Theorem N}
\newtheorem*{ThmO}{Theorem O}
\newtheorem*{ThmP}{Theorem P}
\newtheorem*{ThmQ}{Theorem Q}

\newtheorem{Thm}{Theorem}[section]
\newtheorem{Cor}[Thm]{Corollary}
\newtheorem{Lem}[Thm]{Lemma}
\newtheorem{Pro}[Thm]{Proposition}
\theoremstyle{definition}

\theoremstyle{remark}
\newtheorem{Rem}[Thm]{\upshape\bfseries Remark}
\newtheorem{Ques}[Thm]{Question}
\numberwithin{equation}{section}

\newcommand{\ee}{\mathrm{e}}

\theoremstyle{definition}

\def\be{\begin{equation}}
\def\ee{\end{equation}}

\newcommand{\ben}{\begin{enumerate}}
\newcommand{\een}{\end{enumerate}}

\newcommand{\bca}{\begin{ca}}
\newcommand{\eca}{\end{ca}}

\newcommand{\br}{\begin{rem}}
\newcommand{\er}{\end{rem}}
\newcommand{\brs}{\begin{rems}}
\newcommand{\ers}{\end{rems}}
\newcommand{\bo}{\begin{obser}}
\newcommand{\eo}{\end{obser}}
\newcommand{\bos}{\begin{obsers}}
\newcommand{\eos}{\end{obsers}}
\newcommand{\bpf}{\begin{pf}}
\newcommand{\epf}{\end{pf}}
\newcommand{\ba}{\begin{array}}
\newcommand{\ea}{\end{array}}
\newcommand{\beq}{\begin{eqnarray}}
\newcommand{\beqq}{\begin{eqnarray*}}
\newcommand{\eeq}{\end{eqnarray}}
\newcommand{\eeqq}{\end{eqnarray*}}

\newcommand{\bprop}{\begin{prop}}
\newcommand{\eprop}{\end{prop}}

\numberwithin{equation}{section}
\newcounter{minutes}
\divide\time by 60
\newcounter{hours}
\multiply\time by 60 \addtocounter{minutes}{-\time}
\begin{document}
\title{Sharp Riesz conjugate functions theorems for quasiregular mappings}

\author[Shaolin Chen, Manzi Huang,   Xiantao Wang and Jie Xiao ]{Shaolin Chen, Manzi Huang,   Xiantao Wang and Jie Xiao}

%Shaolin Chen
%\address{S. L.  Chen, College of Mathematics and
%Statistics, Hengyang Normal University, Hengyang, Hunan 421002,
%People's Republic of China; Hunan Provincial Key Laboratory of Intelligent Information Processing and Application,  421002,
%People's Republic of China.} \email{mathechen@126.com}

%\address{H. Hamada, Faculty of Science and Engineering, Kyushu Sangyo University,
%3-1 Matsukadai 2-Chome, Higashi-ku, Fukuoka 813-8503, Japan.}
%\email{ h.hamada@ip.kyusan-u.ac.jp}

%\author{Shaolin Chen}
 %\address{S. L. Chen, College of Mathematics and
%Statistics, Hengyang Normal University, Heng-yang, Hunan 421002,
%People's Republic of China, and Hunan Provincial Key Laboratory of Intelligent Information Processing and Application,  421002,
%People's Republic of China.} \email{mathechen1982@hynu.edu.cn}

\address{S. L. Chen,    Center for Applied Mathematics of Guangxi, Guangxi Normal University,
Guilin, Guangxi 541004, People¡¯s Republic of China} \email{mathechen@126.com}

%\author{Manzi Huang}  %${}^{~\mathbf{*}}$}
\address{M. Z.  Huang, Key Laboratory of High Performance Computing and Stochastic Information Processing,
College of Mathematics and Statistics, Hunan Normal University, Changsha, Hunan 410081, People's Republic of China}
 \email{mzhuang@hunnu.edu.cn}

%\author{Xiantao Wang}
\address{X. T. Wang, Key Laboratory of High Performance Computing and Stochastic Information Processing,
College of Mathematics and Statistics, Hunan Normal University, Changsha, Hunan 410081, People's Republic of China}
\email{xtwang@hunnu.edu.cn}

%\author{Jie Xiao}
\address{J.  Xiao, Department of Mathematics and Statistics, Memorial University, ST. John's, NL
	A1C 5S7, Canada}
\email{jxiao@math.mun.ca}

%{\color{red}}

\maketitle

\def\thefootnote{}
%\footnotetext{* corresponding author}
%\footnotetext{The research of the second named author was supported by NNSF of China Grant Nos. 11501220, 11471128, NNSF of Fujian Province Grant No. 2016J01020,
%and the Promotion Program for Young and Middle-aged Teacher in Science and Technology Research of Huaqiao University (ZQN-PY402).}
\footnotetext{2010 Mathematics Subject Classification. Primary: 30H10, 30C62.}
\footnotetext{Keywords.
Riesz type theorem, $\kappa$-pluriharmonic mappings, harmonic $K$-quasiregular mapping,  invariant harmonic $K$-quasiregular  mapping}
\makeatletter\def\thefootnote{\@arabic\c@footnote}\makeatother

\begin{abstract}
One of the celebrated results by Riesz \cite{Rie} is the Riesz conjugate functions theorem for analytic functions in the complex plane $\mathbb{C}$.
The study on the Riesz conjugate functions theorem for functions in higher dimensional spaces has attracted much attention. Fefferman and Stein \cite{FS-1972} established the Riesz conjugate functions theorem for the Cauchy-Riemann systems in the upper half real space $\mathbb{R}^{n+1}_{+}$. Astala and Koskela \cite{AS-2} investigated the Riesz conjugate functions theorem
for quasiconformal mappings of the unit ball $\mathbf{B}^{n}$ in $\mathbb{R}^n$, and posed an open problem which is as follows: Does there exist a quasiconformal analog for the Riesz theorem on conjugate functions?  The purpose of this paper is to develop some methods to study this topic further, in particular, Astala-Koskela's open problem. First, we prove a sharp Riesz conjugate functions theorem for a class of quasiregular mappings  of $\mathbf{B}^{n}$ for all $n\geq 2$ which satisfy the so-called Heinz's nonlinear differential inequality. As a direct consequence of this result, we find that the answer to Astala-Koskela's open problem is affirmative for harmonic quasiregular mappings of $\mathbf{B}^{n}$ for all $n\geq 2$.
Second, we obtain a sharp  Riesz conjugate functions theorem for invariant harmonic $K$-quasiregular mappings of $\mathbf{B}^{n}$ for all $n\geq 2$ which shows that the answer to Astala-Koskela's open problem is affirmative for these mappings.
At last, we introduce the family of $\kappa$-pluriharmonic mappings of the unit ball $\mathbb{B}^n$ in $\mathbb{C}^n$, and establish a sharp  Riesz conjugate functions theorem for these mappings for all $n\geq 1$. Consequently, we generalize and improve all main results by Liu and Zhu  \cite{L-Z}.
\end{abstract}

\maketitle
\pagestyle{myheadings}
\markboth{S. L. Chen,   M. Z.  Huang,  X. T. Wang and J. Xiao}{Sharp Riesz conjugate functions theorems for quasiregular mappings}

\tableofcontents

\section{Preliminaries and the statement of the main results}\label{csw-sec1}
Let
$\mathbb{C}^{s}$ and
$\mathbb{R}^{\gamma}$ denote the complex space of dimension
$s$ and the real space of dimension
$\gamma$, respectively, where
$s$ and
$\gamma$ are positive integers.
We identify $\mathbb{C}^{s}$ with the
real $2s$-space $\mathbb{R}^{2s}$ if needed. For
$$\Big(z=(z_{1},\ldots,z_{s}), w=(w_{1},\ldots,w_{s})\Big)\in\mathbb{C}^{s}\times \mathbb{C}^{s},
$$
the standard
 Hermitian scalar product of $z$ and $w$, and the
Euclidean norm of $z$ are given by $$\langle z,w\rangle :=
\sum_{j=1}^sz_j\overline{w}_j\;\;\mbox{and}\;\; |z|:={\langle
z,z\rangle}^{1/2}, $$ respectively. For $a\in \mathbb{C}^s$ (resp. $a\in \mathbb{R}^\gamma$ with $\gamma\geq 2$), we use
$\mathbb{B}^s(a,r)$ (resp. $\mathbf{B}^\gamma(a,r)$)
to denote the open ball of radius $r>0$ with center $a$, i.e.,
$$
\mathbb{B}^s(a,r)=\{z\in \mathbb{C}^{s}:\, |z-a|<r\}\ \  (\text{resp. $\mathbf{B}^\gamma(a,r)=\{x\in \mathbb{R}^{\gamma}:\, |x-a|<r\}$}).
$$
Let $\mathbf{S}^{\gamma-1}(a,r):=\partial\mathbf{B}^{\gamma}(a,r)$
be the boundary of $\mathbf{B}^{\gamma}(a,r)$, that is, $$\mathbf{S}^{\gamma-1}(a,r)=\{x\in \mathbb{R}^{\gamma}:\, |x-a|=r\}.
$$
In the following, we identify $\mathbb{B}^\gamma$ with $\mathbf{B}^{2\gamma}$.
Then the boundary of $\mathbb{B}^{s}(a,r)$ is
$$\partial\mathbb{B}^{s}(a,r)=\mathbf{S}^{2s-1}(a,r).$$
For convenience,
let

$$
\begin{cases}
\mathbb{B}^{s}:=\mathbb{B}^{s}(0,1);\\
 \mathbf{B}^{\gamma}_{r}:=\mathbf{B}^{\gamma}(0,r);\\ \mathbf{B}^{\gamma}:=\mathbf{B}^{\gamma}_{1};\\ \mathbf{S}^{\gamma-1}:=\mathbf{S}^{\gamma-1}(0,1);\\
 \mathbb{D}:=\mathbb{B}^{1}=\mathbf{B}^{2}.
 \end{cases}
 $$

\subsection{Hardy type spaces}
For $p\in(0,\infty]$, $n\geq2$ and $m\geq1$, the Hardy type space
$\mathbb{H}^{p}(\mathbf{B}^{n}, \mathbb{R}^{m})$ consists of all functions
$f:~\mathbf{B}^{n}\rightarrow\mathbb{R}^{m}$ such that $f$ is measurable,
$M_{p}(r,f)$ exist for all $r\in[0,1)$ and  $ \|f\|_{p}<\infty$,
where
$$\|f\|_{p}=
\begin{cases}
\displaystyle\sup_{0<r<1} \big\{M_{p}(r,f)\big\}
& \mbox{if } p\in(0,\infty),\\
\displaystyle\sup_{z\in\mathbf{B}^{n}} \big\{ |f(z)| \big\} &\mbox{if } p=\infty,
\end{cases}
$$
 $$M_{p}(r,f)=\left(\int_{\mathbf{S}^{n-1}}|f(r\zeta)|^{p}\,d\sigma(\zeta)\right)^{\frac{1}{p}}$$
and $d\sigma$ denotes the normalized Lebesgue surface measure on
$\mathbf{S}^{n-1}$.

Similarly, we can define the Hardy type space $\mathbb{H}^{p}(\mathbb{B}^{n}, \mathbb{C}^{m})$ for $n\geq 1$ and $m\geq 1$.

For $p\in[1,\infty)$, let $L^{p}(\Omega, dV_{N})$ denote the classical Banach space consisting of all
measurable functions on a bounded domain $\Omega\subseteq\mathbb{R}^{n}$ that are $p$-integrable, where
$dV_{N}$ is the normalized Lebesgue volume measure on $\Omega$.
The norm in $L^{p}(\Omega, dV_{N})$ is defined by
$$\|f\|_{L^{p}(\Omega)}=\left(\int_{\Omega}|f(x)|^{p}dV_{N}(x)\right)^{\frac{1}{p}}.$$

Let us recall one of the celebrated results by Riesz, which is the so-called Riesz conjugate functions theorem $($see \cite{Rie}, or \cite[p.53]{D-1970} and \cite{Ste}$)$.

\begin{ThmA}\label{M-R-1}
Suppose that $u$ is harmonic in $\mathbb{D}$. If $u\in\mathbb{H}^{p}(\mathbb{D}, \mathbb{R})$ for some $p\in(1,\infty)$, then its
harmonic conjugate $v$ is also of class $\mathbb{H}^{p}(\mathbb{D}, \mathbb{R})$ provided that $v(0)=0$. Furthermore, there is a positive constant $C_1$
such that
\be\label{Riez}M_{p}(r,v)\leq C_1 M_{p}(r,u),~r\in[0,1),\ee
for every harmonic $u\in\mathbb{H}^{p}(\mathbb{D}, \mathbb{R})$ and its harmonic conjugate $v$ with $v(0)=0$,
where $C_1=C_1(p)$, which means that the constant $C_1$ depends only on the parameter $p$.
\end{ThmA}

In 1933, Stein gave a new proof of Theorem A, and obtained a new upper bound for the constant $C_1$ (see \cite{Ste}).
The best possibility for $C_1$ in Theorem A was determined by Pichorides (see \cite{Pi}) and, independently, by Cole (cf. \cite{Gam}).
We refer to \cite{AG,CH-2023, Ve} for more discussions on the related studies.

 Observe that Theorem A can be formulated as follows:

\begin{ThmB}\label{M-R-2}
 Suppose that $f=u+iv$ is analytic in $\mathbb{D}$ with $v(0)=0$. There exists a constant $p_{0}=1$ such that if
$u\in\mathbb{H}^{p}(\mathbb{D}, \mathbb{R})$ for some $p>p_{0}$, then $f\in\mathbb{H}^{p}(\mathbb{D}, \mathbb{R}^{2})$.
\end{ThmB}

%a constant $p_{0}=p_{0}(n,K)$ such that if $f_j\in \mathbb{H}^{p}(\mathbf{B}^{n}, \mathbb{R}^{n})$ for some $p>p_{0}$, then $f\in \mathbb{H}^{p}(\mathbf{B}^{n}, \mathbb{R})$.

Inspired by Theorem B, one naturally asks the following question.

\begin{Ques}\label{Question-1.1}
Is there any result similar to Theorem B for functions in $\mathbb{C}^n$ for all $n\geq 1$ or in $\mathbb{R}^m$ for all $m\geq 2$?
\end{Ques}

By generalizing the Cauchy-Riemann equations to the setting of higher dimensional spaces, Stein and Weiss
introduced a class of vector-valued harmonic functions in $(n+1)$-dimensional upper half real space $\mathbb{R}_{+}^{n+1}$ and defined its corresponding Hardy type space which is called the
{\it Stein-Weiss $H^{p}$ space}. See \cite{SW} for the details.
In \cite{FS-1972}, Fefferman and Stein considered Question \ref{Question-1.1} for the class of functions introduced by Stein and Weiss. Their results show that the answer to Question
\ref{Question-1.1} is affirmative for these functions in $\mathbb{R}^{n+1}$. The precise statement of their results is as follows.

\begin{ThmC}\label{F-S-1} {\rm (\cite[lines 3-10 in page 168 and Theorem 9]{FS-1972})}
For $p>(n-1)/n$ and a function $u=(u_{1},u_{2},\ldots,u_{n+1})$ in $\mathbb{R}_{+}^{n+1}$, the following statements are equivalent.
\begin{enumerate}
\item[{\rm (1)}] $u$ belongs to
the Stein-Weiss $H^{p}$ space.
\item[{\rm (2)}] $u_{1}$ belongs to
the Stein-Weiss $H^{p}$ space.
\item[{\rm (3)}] The maximal function
$u^{\ast}(x)=\sup_{|y-x|<t}\{|u_{1}(y,t)|\}$ belongs to $L^{p}(\mathbb{R}^{n})$.
\end{enumerate}
\end{ThmC}

Astala and Koskela considered Question \ref{Question-1.1} for the class of quasiconformal mappings (see \cite{AS-2}).
To state their results, we need to introduce necessary notions and notations.

\subsection{Quasiregular mappings and quasiconformal mappings}
A continuous mapping $f:\ \Omega\subset\mathbb{R}^{n}\rightarrow\mathbb{R}^{n}$ is called {\it $K$-quasiregular} if $f\in
W_{n,{\rm loc}}^{1}(\Omega)$ and
$$|f'(x)|^{n}\leq KJ_{f}(x)
$$ for almost every $x\in\Omega$,
where $K$ $(\geq1)$ is a constant, $f\in W_{n,{\rm loc}}^{1}(\Omega)$ means that the
distributional derivatives $\partial f_{j}/\partial x_{k}$ of the
coordinates $f_{j}$ of $f $ are locally in $L^{n}(\Omega)$ and $J_{f}(x)$
denotes the Jacobian of $f$ (cf. \cite{V}). In particular, if a $K$-quasiregular mapping $f$ is homeomorphic, then
$f$ is a $K$-quasiconformal mapping (see \cite{Va}). We refer to \cite{D-P,Rick} for more details on $K$-quasiregular mappings.

By applying the growth theorem and the distortion theorem of quasiconformal mappings, Astala and Koskela obtained the following.

\begin{ThmD}\label{A-K}{\rm (\cite[Theorem  6.2]{AS-2})}
Suppose that $f$ is $K$-quasiconformal in $\mathbf{B}^{n}$ $(n\geq2)$ with one of its coordinate functions belongs to
$\mathbb{H}^{p}(\mathbf{B}^{n}, \mathbb{R}^{n})$ for some $p>0$. Then $f\in\mathbb{H}^{q}(\mathbf{B}^{n}, \mathbb{R}^{n})$ for all $q<p$.
\end{ThmD}

Meanwhile, in \cite{AS-2}, Astala and Koskela constructed a $K$-quasiconformal mapping to show that $q$ cannot reach $p$,
where $K=K(n,p)$.
Since $K$ depends on the parameter $p$,
we see that this example does not show that the answer to Question \ref{Question-1.1} is negative for quasiconformal mappings.
In order to thoroughly understand this issue for $K$-quasiconformal mappings, Astala and Koskela asked whether
there exists a quasiconformal analog for the Riesz theorem on conjugate functions. More precisely,
given $n\geq2$ and $K\geq1$, does there exist $p_0=p_0(n,K)$ such that each
$K$-quasiconformal mapping whose first coordinate function belongs to $\mathbb{H}^{p}(\mathbf{B}^{n}, \mathbb{R})$
for some $p>p_{0}$, in fact, belongs itself to $\mathbb{H}^{p}(\mathbf{B}^{n}, \mathbb{R}^{n})$ (see Open problem (1) in \cite{AS-2})?
Motivated by this open problem, we raise the following question concerning quasiregular mappings.

\begin{Ques}\label{Question-1.2}
Let $n\geq2$ and $K\geq1$. Suppose $f=(f_{1},\ldots,f_{n})$ is a $K$-quasiregular mapping from
$\mathbf{B}^{n}$ into $\mathbb{R}^{n}$. Determine a sufficient condition on $f$ ensuring that there exist an
index $j\in \{1, \ldots, n\}$ and a constant $p_{0}=p_{0}(n,K)$ such that if $f_j\in \mathbb{H}^{p}(\mathbf{B}^{n}, \mathbb{R}^{n})$ for some $p>p_{0}$, then $f\in \mathbb{H}^{p}(\mathbf{B}^{n}, \mathbb{R})$.

%Given $n\geq2$ and $K\geq1$, let $f=(f_{1},\ldots,f_{n})$ be a $K$-quasiregular mapping of $\mathbf{B}^{n}$ into $\mathbb{R}^{n}$.
%Find a suitable condition on $f$ to guarantee that there exist some $j\in \{1, \ldots, n\}$ and some constant $p_{0}=p_{0}(n,K)$ %such that
%$f_j\in \mathbb{H}^{p}(\mathbf{B}^{n}, \mathbb{R}^{n})$ for some $p>p_{0}$ implies that $f\in \mathbb{H}^{p}(\mathbf{B}^{n}, %\mathbb{R})$.
\end{Ques}

The purpose of this paper is to discuss Questions \ref{Question-1.1} and \ref{Question-1.2}. Our results are Theorems \ref{lem-im-1}, \ref{thm-4} and  \ref{thm-1} below.
For their precise statements, some preparation is needed. We start the preparation with the definition of harmonic functions.

\subsection{Harmonic functions}
A twice continuously differentiable  function $u$ of a domain $D\subset\mathbb{R}^{m}$ into $\mathbb{R}:=\mathbb{R}^{1}$ is said to be
{\it harmonic} if $$\Delta\,u(x)=0\ \ \forall\ \ x=(x_{1},\ldots,x_{m})\in D,
$$
where
$$\Delta:=\sum_{j=1}^{m}\frac{\partial^{2}}{\partial\,x^{2}_{j}}$$ is the usual Laplacian operator (see \cite{ABR-2001}).
Furthermore, a twice continuously differentiable  complex-valued function $f=u+iv$ of a domain $\Omega\subset\mathbb{C}^{n}$ into $\mathbb{C}:=\mathbb{C}^{1}$ is said to be
{\it harmonic} if not only $$\Delta\,u(z)=\sum_{j=1}^{n}\left(\frac{\partial^{2} u(z)}{\partial x_{j}^{2}}+\frac{\partial^{2} u(z)}{\partial y_{j}^{2}}\right)=4\sum_{j=1}^{n}\frac{\partial^{2}u(z)}{\partial z_{j}\partial\overline{z}_{j}}=0$$
but also
$$\Delta\,v(z)=\sum_{j=1}^{n}\left(\frac{\partial^{2} v(z)}{\partial x_{j}^{2}}+\frac{\partial^{2} v(z)}{\partial y_{j}^{2}}\right)
=4\sum_{j=1}^{n}\frac{\partial^{2}v(z)}{\partial z_{j}\partial\overline{z}_{j}}=0,$$
where
$$
z=(z_{1},\ldots,z_{n})=(x_{1}+iy_{1},\ldots,x_{n}+iy_{n})\in\Omega.
$$
Moreover, if $n=1$ and $\Omega\subset\mathbb{C}$ is a simply connected domain, then a complex-valued harmonic function $f:\,\Omega\rightarrow
\mathbb{C}$  has a representation
$f=h+\overline{g},$ where $h$ and $g$ are holomorphic in $\Omega$. This representation is unique up to an additive constant (see \cite{Du-2}).

For convenience, in the following, we use ${\rm Re}(f)$ (resp. ${\rm Im}(f)$) to denote the real (resp. imaginary) part function of $f$. This implies that ${\rm Re}(f)=u$ and ${\rm Im}(f)=v$.

For a $K$-quasiregular (resp. quasiconformal) mapping $f=(f_{1},\ldots,f_{n})$   of
 $\mathbf{B}^{n}$ into $\mathbb{R}^{n}$, we call that $f$ is a {\it harmonic $K$-quasiregular} (resp. {\it harmonic $K$-quasiconformal}) mapping if
 for each $j\in\{1,\ldots,n\}$, $f_{j}$ is a harmonic function. In particular,
for a sense-preserving harmonic function in $\mathbb{D}$, its quasiregularity can be equivalently defined in the following way.
Given $K\geq1$, a sense-preserving harmonic function $f=h+\overline{g}$ in $\mathbb{D}$ is said to be {\it $K$-quasiregular} if
\be\label{KQR}|g'(z)|\leq \kappa|h'(z)|\ \ \text{with}\ \ \kappa=(K-1)/(K+1),\ee
where $h$ is a locally biholomorphic function and $g$ is a holomorphic function.
Also, $f$ is said to be {\it $K$-quasiconformal}
 in $\mathbb{D}$ if it is $K$-quasiregular and homeomorphic (see \cite{Du-2,L-Z}).
Recall that a harmonic function $f=h+\overline{g}$ is
sense-preserving in $\mathbb{D}$ if and only if $|h'|>|g'|$ in $\mathbb{D}$. If $f=h+\overline{g}$ is sense-preserving, then $h$ must be locally biholomorphic.

Very recently, Liu and Zhu \cite{L-Z}
considered Question \ref{Question-1.2}
 for the planar harmonic quasiregular
(or quasiconformal) mappings. Their results are as follows.

\begin{ThmE}{\rm  (\cite[Theorem  1.1]{L-Z})}\label{Thm-P}
Suppose that $f=u+iv$  is a planar harmonic $K$-quasiregular mapping with $v(0)=0$ and $u\geq0$ $($or $u<0$$)$. If $u\in \mathbb{H}^{p}(\mathbb{D},\mathbb{R})$ for some $p\in(1,2]$, then $v\in \mathbb{H}^{p}(\mathbb{D},\mathbb{R}).$
Furthermore, there is a positive constant $C_2$ such that
$$M_{p}(r,v)\leq C_2 M_{p}(r,u),$$ where %$K=(1+\kappa)/(1-\kappa)$ and
$C_2=C_2(K,p)$.
Moreover, if $K=1$, then $$C_2(1,p)=\cot(\pi/(2p^{\ast})),$$ which coincides with the classical analytic case.
\end{ThmE}

\begin{ThmF}{\rm  (\cite[Theorem  1.2]{L-Z})}\label{Thm-Q}
Suppose that $f=u+iv$ is a planar harmonic $K$-quasiconformal mapping with $v(0)=0$.
If $u\in \mathbb{H}^{p}(\mathbb{D},\mathbb{R})$ for some $p\in(2,\infty)$, then $v\in \mathbb{H}^{p}(\mathbb{D},\mathbb{R}).$
Furthermore, there is a positive constant $C_3$ such that
$$M_{p}(r,v)\leq C_3M_{p}(r,u),$$ where $C_3=C_3(K,p)$.
%Moreover, if $K=1$, then $$C_3(1,p)=\cot(\pi/(2p)),$$ which coincides with the classical analytic case.
\end{ThmF}

Since Theorems E and F does not resolve whether Question \ref{Question-1.2} holds for plane harmonic mappings, we proceed to investigate this problem in the subsequent analysis.
Specifically, Theorems E and F inspire us to consider the following natural problem.

\begin{Ques}\label{Question-1.4}
Given $n\geq2$ and $K\geq1$, does there exist $p_0=p_0(n,K)$ such that each harmonic
$K$-quasiregular mapping $f=(f_{1},\ldots,f_{n})$ whose first coordinate function belongs to $\mathbb{H}^{p}(\mathbf{B}^{n}, \mathbb{R})$
for some $p>p_{0}$, in fact, belongs itself to $\mathbb{H}^{p}(\mathbf{B}^{n}, \mathbb{R}^{n})$?
\end{Ques}

The first purpose of this paper is to investigate Questions \ref{Question-1.2} and \ref{Question-1.4}. Before presenting
our first main result, let us do some preparation.

\subsection{Quasiregular mappings satisfying the Heinz's nonlinear
differential inequality}
For $n\geq2$,
a non-negative locally integrable function $\psi$ in a domain $\Omega\subset\mathbb{R}^{n}$ is said to be (Euclidean) {\it quasi-nearly subharmonic} if there is a positive constant
$C$ such that
$$\psi(x)\leq\frac{C}{r^{n}}\int_{\mathbf{B}^n(x,r)}\psi(y)dV_{N}(y)$$
 holds for any ball $\mathbf{B}^n(x,r)\subset\Omega$, where $C=C(n,\Omega)$ (see \cite{FS-1972,P-R,Sto}).

It is well-known that every subharmonic function is quasi-nearly subharmonic (see \cite{P-R,Sto}).
Quasi-nearly subharmonic functions, perhaps with a different terminology, have been considered by many authors.
The concept itself dates back to Fefferman and Stein (see \cite{FS-1972}), who proved that if $f$ is harmonic in $\mathbf{B}^{n}$, then
$|f|^{p}$ is quasi-nearly subharmonic for all $p\in(0,1)$. This result was also proved independently by Kuran (see \cite{Kur}). For
a nonnegative subharmonic function $\zeta$,  Riihentaus (see \cite{Rii}) and Suzuki (see \cite{Suz}) independently proved that
$\zeta^{p}$ is quasi-nearly subharmonic for all $p\in(0,1)$.

Let $\varphi$ be a twice continuously differentiable function of $\mathbf{B}^{n}$ into $\mathbb{R}$. If there exist a
real-valued nonnegative continuous function $\chi$ in $\mathbf{B}^{n}$ and
nonnegative constants $\gamma_{j}~(j\in\{1,2,3\})$ such that
\be\label{Heinz-57}|\varphi(x)|^{\gamma_{1}}|\Delta \varphi(x)|^{\gamma_{2}}\leq\chi(x)|\nabla \varphi(x)|^{\gamma_{3}},\ee
then we say that $\varphi$ satisfies the {\it Heinz type nonlinear
differential inequality}. By using (\ref{Heinz-57}), Heinz investigated the  mean curvature of surfaces,
the elliptic Monge-Amp${\rm\grave{e}}$re equations, the Poisson equations, the gradient equations, the Harnack inequality for a class of elliptic
differential equations and the existence of solutions to some class of elliptic
differential equations (see \cite{HZ}). This remarkable paper has attracted much attention (see, e.g., \cite{Mar-,Ni,Qiu,Sau,S-Y,Wan}).

For integers $n\geq2$ and $m\geq1$,
denote by
$\mathscr{H}_{a}(\mathbf{B}^{n},\mathbb{R}^{m})$ the class of all functions
$$f=(f_{1},\ldots,f_{m})\in\,C^{2}(\mathbf{B}^{n})$$ satisfying {\it the Heinz's nonlinear
differential inequality}
\beqq\label{eq-15}
0 \leq\sum_{j=1}^{m}f_{j}(x)\Delta f_{j}(x)\leq a(x)\left(\sum_{j=1}^{m}|\nabla f_{j}(x)|^{2}\right)\ \ \forall\ \ x\in\mathbf{B}^n, \eeqq
where $a$ is a
real-valued nonnegative continuous function in $\mathbf{B}^{n}$.

The first purpose of this paper is to give an answer to Question \ref{Question-1.2}. Our first result statement
is as follows.

\begin{Thm}\label{lem-im-1}
Let $n\geq2$, and let $$f=(f_{1},\ldots,f_{n})\in\mathscr{H}_{a}(\mathbf{B}^{n},\mathbb{R}^{n})$$ be a $K$-quasiregular mapping
such that $\sup_{x\in\mathbf{B}^{n}}\{a(x)\}<\infty.$ Suppose there exists some $j\in\{1,\ldots,n\}$ for which:
\begin{enumerate}
\item[{\rm (1)}] $f_{{j}}\Delta f_{{j}}\geq0$, and
\item[{\rm (2)}] $\Delta f_{{j}}^{2}$ is quasi-nearly subharmonic in $\mathbf{B}^{n}$.
\end{enumerate}
If  $f_{{j}}\in \mathbb{H}^{p}(\mathbf{B}^{n}, \mathbb{R})$ for some $p>p_{0}=1$, then
$f\in\mathbb{H}^{p}(\mathbf{B}^{n}, \mathbb{R}^{n})$. Moreover, the constant $p_{0}=1$ is sharp.

%For $n\geq2$, let $f=(f_{1},\ldots,f_{n})\in\mathscr{H}_{a}(\mathbf{B}^{n},\mathbb{R}^{n})$ be a $K$-quasiregular mapping with %$\sup_{x\in\mathbf{B}^{n}}\{a(x)\}<\infty$. Suppose that there is a ${k}\in\{1,\ldots,n\}$ such that $f_{{k}}\Delta f_{{k}}\geq0$ %and $\Delta f_{{k}}^{2}$ is quasi-nearly subharmonic in $\mathbf{B}^{n}$. Then there exist a $p_0=1$ and a function $f_{k}$ with %$k\in \{1, \cdots, n\}$ such that
%$f_{{k}}\in \mathbb{H}^{p}(\mathbf{B}^{n}, \mathbb{R})$ for some $p>p_{0}$ implies $f\in\mathbb{H}^{p}(\mathbf{B}^{n}, %\mathbb{R}^{n})$.
\end{Thm}

Obviously, every harmonic function belongs to $\mathscr{H}_{a}(\mathbf{B}^{n},\mathbb{R}^{m})$ for any real-valued nonnegative continuous function $a$ in $\mathbf{B}^{n}$.
This ensures that the following Riesz conjugate functions theorem is a direct consequence of Theorem \ref{lem-im-1}, which implies that the answer to
 Question \ref{Question-1.4} is affirmative for the class of harmonic $K$-quasiregular mappings of $\mathbf{B}^{n}$ into $\mathbb{R}^{n}$. It is also an improvement and
 generalization of Theorems E and F.

\begin{Cor}\label{cor-top1}
For $n\geq2$, let $f=(f_{1},\ldots,f_{n})$ be a harmonic $K$-quasiregular mapping of $\mathbf{B}^{n}$ into $\mathbb{R}^{n}$.
There exist an
index $k\in \{1, \ldots, n\}$ and a constant $p_{0}=1$ such that if $f_k\in \mathbb{H}^{p}(\mathbf{B}^{n}, \mathbb{R}^{n})$ for some $p>p_{0}$, then $f\in\mathbb{H}^{p}(\mathbf{B}^{n}, \mathbb{R}^{n})$. Furthermore, the constant $p_{0}=1$ is sharp.
%Then there  is $k\in\{1,\ldots,n\}$ such that
%$f_{k}\in \mathbb{H}^{p}(\mathbf{B}^{n}, \mathbb{R})$ for some $p>1$ if and only if $f\in\mathbb{H}^{p}(\mathbf{B}^{n}, \mathbb{R}^{n})$.
Moreover, we have the following estimates:
\begin{enumerate}
\item[{\rm ($\mathscr{A}_{1}$)}]\quad For $p\in(1,2]$, if
$$
\exists\ \text{some $k\in\{1,\ldots,n\}$ such that $f_{k}\in \mathbb{H}^{p}(\mathbf{B}^{n}, \mathbb{R})$},
$$
then for $r\in[0,1)$, \be\label{ie-1}M_{p}^{p}(r,f)\leq|f(0)|^{p}+\left(\frac{1+(n-1)K^{2}}{p-1}\right)\left(M_{p}^{p}(r,f_{k})-|f_{k}(0)|^{p}\right).\ee
\item[{\rm ($\mathscr{A}_{2}$)}]\quad For $p\in(2,\infty)$, if
$$
\exists\ \text{some $k\in\{1,\ldots,n\}$ such that $f_{k}\in \mathbb{H}^{p}(\mathbf{B}^{n}, \mathbb{R})$},
$$
then
 there is a positive constant $\tau_1$ such that
\be\label{ie-2}\|f\|_{p}^{2}\leq|f(0)|^{2}+\tau_1\|f_{k}\|_{p}^{2},\ee where $\tau_1=\tau_1(n,p,K)$.
\end{enumerate}
\end{Cor}

We remark that the inequalities (\ref{ie-1}) and (\ref{ie-2}) in Corollary \ref{cor-top1} follow from (\ref{mm-1}) and (\ref{mm-2}) below, respectively.

\subsection{Invariant harmonic quasiregular mappings}
 For $n\geq2$ and a twice continuously differentiable function $f$ of $\mathbf{B}^{n}$ into $\mathbb{R}$, let
$$\Delta_{h}f=(1-|x|^{2})^{2}\Delta f+2(n-2)(1-|x|^{2})\langle \nabla f,x\rangle,$$ where $$x=(x_1,\ldots, x_n)\in \mathbf{B}^{n}\ \ \&\ \
\nabla
=
\left(
\frac{\partial }{\partial x_1},
\dots,
\frac{\partial }{\partial x_n}
\right).
$$
The operator $\Delta_{h}$ is called the {\it invariant Laplacian} or {\it Laplace-Beltrami} operator on $\mathbf{B}^{n}$.
If $$\Delta_{h}f=0$$ holds in $\mathbf{B}^{n}$, then $f$ is said to be
{\it invariant harmonic} (see \cite{Ru-1,Sto}).

For a function $\varphi\in L^{1}(\mathbf{S}^{n-1})$, we denote by $P_{h}[\varphi]$ the Dirichlet solution, for the invariant Laplacian operator
$\Delta_{h}$, of $\varphi$ over $\mathbf{B}^{n}$, that is $\Delta_{h} P_{h}[\varphi]=0$ in $\mathbf{B}^{n}$ and $P_{h}[\varphi]=\varphi$ on $\mathbf{S}^{n-1}$.
It is well-known that, for  $\varphi\in L^{1}(\mathbf{S}^{n-1})$,

$$P_{h}[\varphi](x)=\int_{\mathbf{S}^{n-1}}\mathbf{P}_{h}(x,\zeta)\varphi(\zeta)d\sigma(\zeta),$$
where $\mathbf{P}_{h}(x,\zeta)=\frac{(1-|x|^{2})^{n-1}}{|x-\zeta|^{2n-2}}$
is the {\it hyperbolic Poisson kernel} (see \cite{Kw,Ru-1,Sto}).

For a quasiregular (resp. quasiconformal) mapping $f=(f_{1},\ldots,f_{n})$ of
 $\mathbf{B}^{n}$ into $\mathbb{R}^{n}$, we call that $f$ is an {\it invariant harmonic quasiregular} (resp. {\it invariant harmonic quasiconformal}) mapping if
 for each $j\in\{1,\ldots,n\}$, $f_{j}$ is an invariant harmonic function.

In \cite{L-Z}, Liu and Zhu considered Question \ref{Question-1.2} for invariant harmonic quasiregular mappings. The following are their results.

\begin{ThmG}{\rm  (\cite[Theorem  1.3]{L-Z})}\label{L-Z-C}
For $n\geq2$ and $K\geq1$, suppose that $f=(f_{1},\ldots,f_{n})$ is an invariant harmonic $K$-quasiregular  mapping of $\mathbf{B}^{n}$ such that
$f_{1}\neq0$ in $\mathbf{B}^{n}$, where for each $j\in\{1,\ldots,n\}$, $f_{j}$ is an invariant harmonic function of $\mathbf{B}^{n}$ into $\mathbb{R}$. If $f_{1}\in\mathbb{H}^{p}(\mathbf{B}^{n}, \mathbb{R})$ for some $p\in(1,2]$, then
$f\in\mathbb{H}^{p}(\mathbf{B}^{n}, \mathbb{R}^{n})$ with
$$\|f\|_{p}^{p}\leq|f(0)|^{p}+\left(\frac{1+(n-1)K^{2}}{p-1}\right)\left(\|f_{1}\|_{p}^{p}-|f_{1}(0)|^{p}\right).$$
\end{ThmG}

\begin{ThmH}{\rm  (\cite[Theorem  1.4]{L-Z})}\label{L-Z-D}
For $n\geq2$ and  $K\geq1$, suppose that $f=(f_{1},\ldots,f_{n})$ is an invariant harmonic $K$-quasiregular mapping of $\mathbf{B}^{n}$ such that
 $f_{1}\in\mathbb{H}^{p}(\mathbf{B}^{n}, \mathbb{R})$, where $p\in(n,\infty)$ and,  for each $j\in\{1,\ldots,n\}$,
 $f_{j}$ is an invariant harmonic function of $\mathbf{B}^{n}$ into $\mathbb{R}$. Then $f\in L^{p}(\mathbf{B}^{n}, dV)$ and
 $$M_{p}(r,f)\leq{(1-r^{2})^{-\frac{1}{pq}}}\left(C_4|f(0)|^{p}+C_5\|f_{1}\|_{p}^{p}\right)^{\frac{1}{p}},$$
 where
 $$
 \begin{cases} C_4=C_4(n,p,K);\\
 C_5=C_5(n,p,K);\\
 q=p/(p-1).
 \end{cases}
 $$
\end{ThmH}

As the second purpose of this paper, we investigate Question \ref{Question-1.2} for invariant harmonic quasiregular mappings further. Our result reads as follows.

\begin{Thm}\label{thm-4}
For $n\geq2$ and  $K\geq1$, suppose that $f=(f_{1},\ldots,f_{n})$ is an invariant harmonic $K$-quasiregular mapping
 of $\mathbf{B}^{n}$ into $\mathbb{R}^n$.
 There exist an
index $j\in \{1, \ldots, n\}$ and a constant $p_{0}=1$ such that if $f_j\in \mathbb{H}^{p}(\mathbf{B}^{n}, \mathbb{R}^{n})$ for some $p>p_{0}$, then $f\in\mathbb{H}^{p}(\mathbf{B}^{n}, \mathbb{R}^{n})$.
Furthermore,  the constant $p_{0}=1$ is sharp.

 %Then
%there exists $j\in\{1,\ldots,n\}$ such that
%$f_{j}\in \mathbb{H}^{p}(\mathbf{B}^{n}, \mathbb{R})$ for some $p>1$ if and only if $f\in\mathbb{H}^{p}(\mathbf{B}^{n}, \mathbb{R}^{n})$.

Moreover, we have the following estimates:
\begin{enumerate}
\item[{\rm ($\mathscr{B}_{1}$)}]\quad For $p\in(1,2]$, if $$\exists\ j\in\{1,\ldots,n\}\ \text{such that $f_{j}\in \mathbb{H}^{p}(\mathbf{B}^{n}, \mathbb{R})$},
$$
then $$\|f\|_{p}^{p}\leq|f(0)|^{p}+\left(\frac{1+(n-1)K^{2}}{p-1}\right)\left(\|f_j\|_{p}^{p}-|f_j(0)|^{p}\right).$$
\item[{\rm ($\mathscr{B}_{2}$)}]\quad For $p\in(2,\infty)$, if
$$
\exists\ j\in\{1,\ldots,n\}\ \text{such that $f_{j}\in \mathbb{H}^{p}(\mathbf{B}^{n}, \mathbb{R})$},
$$
then there is {a positive constant $\tau_2=\tau_2(n,p)$ such that}
$$
\|f\|_{p}\leq \big(\tau_2 K(n-1)+1\big)\|f_j\|_{p}-|f_j(0)|+\sqrt{n}|f(0)|.
$$
\end{enumerate}
\end{Thm}

\begin{Rem}  Theorem \ref{thm-4} shows that the answer to Question \ref{Question-1.2} is affirmative for the class of invariant harmonic $K$-quasiregular mappings
 of $\mathbf{B}^{n}$ into $\mathbb{R}^n$. Obviously, Theorem \ref{thm-4} is a unified form of Theorems G and  H, the rest two main results by Liu and Zhu \cite{L-Z}.
By comparing with Theorems G and H, Theorem \ref{thm-4} has further advantages as follows.
\begin{enumerate}
\item[$(i)$]
Theorem \ref{thm-4} shows that the assumption of ``$f_{1}\neq 0$" in Theorem G is redundant.

\item[$(ii)$]
The union of the ranges of $p$ discussed in Theorem G and Theorem H is $(1,2]\cup (n,+\infty)$. Obviously, when $n\geq 3$, $(1,2]\cup (n,+\infty)$ is a proper subset of $(1,\infty)$. But, in Theorem \ref{thm-4}, the range of $p$ is the whole interval $(1,\infty)$.

\item[$(iii)$]
Strictly speaking, Theorem H is not a Riesz conjugate functions theorem
 for the class of invariant harmonic $K$-quasiregular mappings. This is because the implication
 $$\text {from}\ \ f_{1}\in\mathbb{H}^{p}(\mathbf{B}^{n}, \mathbb{R})\ \text{to}\ \ f\in\mathbb{H}^{p}(\mathbf{B}^{n}, \mathbb{R}^{n})
 $$
 is not guaranteed by Theorem H. However, Theorem \ref{thm-4} overcomes this deficiency.
\end{enumerate}
\end{Rem}

\subsection{Pluriharmonic functions}
In particular, a twice continuously differentiable function $f$ of
$\Omega$ into $\mathbb{C}$ is called a {\it
pluriharmonic function} if for $z\in \Omega$ and
$\theta\in\mathbf{S}^{2n-1}$, the function
$f(z+\zeta\theta)$ with respect to the variable $\zeta$ is harmonic on $\{\zeta\in \mathbb{C}:\; |\zeta|<
d_{\Omega}(z)\}$,
where $d_{\Omega}(z)$ is the distance from
$z$ to the boundary $\partial\Omega$ of $\Omega$ (see \cite{DHK-2011, Ru,Ru-1, Vl}).

Moreover, if $\Omega$ is a simply connected domain, then a pluriharmonic function $f:\,\Omega\rightarrow
\mathbb{C}$  has a representation
$f=h+\overline{g},$ where $h$ and $g$ are holomorphic in $\Omega$. This representation is unique up to an additive constant (see \cite{CH-2022,CHPV,DHK-2011,Vl}). Furthermore, from this representation, it is easy to know that pluriharmonic
 functions are generalizations of holomorphic functions.
A vector-valued function $f=(f_{1},\ldots,f_{\nu})$ defined in $\Omega$ is said to be {\it pluriharmonic} if each coordinate function $f_{j}$
($1\leq j\leq \nu$) of $f$ is pluriharmonic. When $n=\nu=1$, pluriharmonic functions are the {\it planar harmonic} functions. Obviously, planar harmonic functions are generalizations of analytic functions. See \cite{Du-2} for more properties of planar harmonic functions.

For convenience, throughout this paper, we use $\mathscr{H}(\Omega, \mathbb{C}^{\nu})$ (resp.
$\mathscr{PH}(\Omega,\mathbb{C}^{\nu})$) to denote the set of all holomorphic functions (resp. all pluriharmonic functions) of $\Omega$ into $\mathbb{C}^{\nu}$.
In particular, we use $$\mathbf{h}^{p}(\mathbb{B}^{n}, \mathbb{C}^{\nu})=\mathbb{H}^{p}(\mathbb{B}^{n}, \mathbb{C}^{\nu})\cap\mathscr{PH}(\mathbb{B}^{n},\mathbb{C}^{\nu})
 ~\mbox{and}~
\mathbf{H}^{p}(\mathbb{B}^{n},\mathbb{C}^{\nu})=\mathbb{H}^{p}(\mathbb{B}^{n}, \mathbb{C}^{\nu})\cap \mathscr{H}(\mathbb{B}^{n},\mathbb{C}^{\nu})$$
to denote the {\it pluriharmonic Hardy} space and the {\it holomorphic Hardy} space, respectively (see \cite{D-1970,Ru-1,Zh}).

If $f\in\mathbf{h}^{p}(\mathbb{B}^{n},\mathbb{C}^{\nu})$ for some $p\in [1,\infty)$,
then the radial limit
$$f(\zeta)=\lim_{r\rightarrow1^{-}}f(r\zeta)$$ exists for almost every $\zeta\in\partial\mathbb{B}^{n}$ (see \cite[Theorems 6.7,  6.13 and 6.39]{ABR-2001}).
Moreover, if $p\in[1,\infty)$, then $\mathbf{h}^{p}(\mathbb{B}^{n},\mathbb{C}^{v})$ is a normed space
with respect to the norm $\| \cdot\|_p$.
%and if $p\in(0,1)$, then $\mathbf{h}^{p}(\mathbb{B}^{n},\mathbb{C}^{v})$ is a metric space with respect to the metric $d(f,g)=\| f-g\|_p^p$.

\subsection{The second Beltrami type equation}

For
$$f=(f_{1},\ldots,f_{\nu})\in\mathscr{PH}(\Omega,\mathbb{C}^{\nu}),
$$
we denote by $\partial
f/\partial z_{j}$ (resp. $ \partial
f/\partial \overline{z}_{j}$) the column vector formed by the partial derivatives of its coordinate
functions, namely,
$$
\partial
f_{1}/\partial z_{j},\ldots ,\partial f_{\nu}/\partial z_{j}\ (\text{resp.}\ \partial
f_{1}/\partial \overline{z}_{j},\ldots ,\partial f_{\nu}/\partial \overline{z}_{j}),
$$
and let
$$Df=\left (\frac{\partial f}{\partial z_{1}}~\ldots ~\frac{\partial f}{\partial z_{n}}\right ) :=
\left ( \frac{\partial f_k}{\partial z_{j}}\right )_{\nu\times n}$$ $$\left(\mbox{resp.}~ \overline{D}f=\left
(\frac{\partial f}{\partial \overline{z}_{1}}~\ldots ~\frac{\partial f}{\partial \overline{z}_{n}}\right ) :=
\left ( \frac{\partial f_k}{\partial \overline{z}_{j}}\right )_{\nu\times n}\right).
$$

For a $\nu\times n$ complex matrix $A=(a_{kj})$,
the {\it Frobenius norm} of $A$ is defined as follows:
\[
\| A\|_F=
\sqrt{\sum_{k=1}^\nu\sum_{j=1}^n |a_{kj}|^2}.
\]
Then we have
\begin{equation}
\label{A-Frobenius-Operator}
\| A\|^2\leq\| A\|_F^2
\leq
n\| A\|^2,
\end{equation}
where
{\[
\|A\|=\sup_{\xi\in \mathbb{C}^n\backslash\{0\}}\left\{\frac{\|A
\xi\|}{\|\xi\|}\right\}.
\]}

Suppose that
$$
\begin{cases}
f\in W_{2n,{\rm loc}}^{1}(\mathbb{B}^{n});\\
\omega=\big(a_{kj}\big)_{n\times n};\\
\|\omega\|=\sup_{z\in \mathbb{B}^{n},\; \xi\in \mathbb{C}^n\backslash\{0\}}\left\{\frac{|\omega(z)
\xi|}{|\xi|}\right\}<1,
\end{cases}
$$
where $\|\omega\|$ is called the {\it operator norm}  of $\omega$, and all $a_{kj}$ are holomorphic functions
of $\mathbb{B}^{n}$ into $\mathbb{C}$. The following equation is called the {\it second Beltrami type} equation:
\be\label{Bel-1}
\overline{\overline{D}f}=\omega_f\, Df.
\ee Also, $\omega_f$ is called the {\it second dilation} of $f$.

It is not difficult to know that every sense-preserving mapping $f\in \mathscr{PH}(\mathbb{B}^{n},\mathbb{C}^{n})$ satisfies the equation \eqref{Bel-1}.
When $n=1$, the second Beltrami type equation is exactly the second Beltrami equation.
In particular, if $n=1$ and $\omega_{f}\equiv 0$, then
$\overline{D}f\equiv0$ which is the well-known Cauchy-Riemann equation. Therefore, the second Beltrami type equation is a generalization
of the Cauchy-Riemann equation.

 %In fact, Poletsky showed
%that holomorphic quasiregular mappings in bounded domains are rather rigid (see \cite{Pol}).

\subsection{$\kappa$-pluriharmonic mappings}

Let $\kappa\in[0,1)$ be a constant. A function $f\in \mathscr{PH}(\mathbb{B}^{n},$ $\mathbb{C}^{n})$ is called a {\it $\kappa$-pluriharmonic} mapping if $\det Df\neq0$ and its second dilation satisfies
$$\left\|\omega_{f}\right\|\leq \kappa.
$$

In the following, we use $\mathscr{PH}_{n}(\kappa)$ to denote the set of all $\kappa$-pluriharmonic mappings $f=u+iv$ with $v(0)=0$, where $u$ and $v$
are pluriharmonic functions of $\mathbb{B}^{n}$ into $\mathbb{R}^{n}$. It follows from \cite[Theorem 5]{DHK-2011} that all $f\in\mathscr{PH}_{n}(\kappa)$ are sense-preserving.

Since $\mathbb{B}^{n}$
is a simply connected domain, we see that each $f\in\mathscr{PH}_{n}(\kappa)$ admits a decomposition
$f = h + \overline{g}$, where $h$ is a locally biholomorphic function and $g$ is a holomorphic function.
The following are relations between $\kappa$-pluriharmonic mappings and harmonic quasiregular mappings.

\begin{Pro}\label{Prop-1}
Suppose that $f=h+\overline{g}$, where $h$ is a locally biholomorphic function and $g$ is a holomorphic function in $\mathbb{B}^{n}$.
\begin{enumerate}
\item[{\rm (1)}]  If $n=1$, then $f\in\mathscr{PH}_{1}(\kappa)$ if and only if it is a planar  harmonic $K$-quasiregular mapping, where $K=(1+\kappa)/(1-\kappa)$.
\item[{\rm (2)}] If $n\geq2$, then for $f = h + \overline{g}\in\mathscr{PH}_{n}(\kappa)$, $f$ is $K$-quasiregular if and only if $h$ is $K^*$-quasiregular, where $K$ and $K^*$ depend on each other, possibly with $\kappa$ and $n$.
\item[{\rm (3)}]   For $n\geq 2$, $\mathscr{PH}_{n}(\kappa)$ contains elements which are not $K$-quasiregular for any $K\geq 1$.
\end{enumerate}
\end{Pro}

\begin{Rem}
When $n=1$, Proposition \ref{Prop-1}{\rm (1)} shows that $\mathscr{PH}_{1}(\kappa)$ is exactly the set of all planar harmonic $K$-quasiregular mappings, where $K=(1+\kappa)/(1-\kappa)$.
However, when $n\geq 2$, the situation becomes different since the elements of $\mathscr{PH}_{n}(\kappa)$  may not necessarily be quasiregular.
\end{Rem}

As the last purpose of this paper, we study Question \ref{Question-1.1} for $\kappa$-pluriharmonic mappings. Our result is as follows.

\begin{Thm}\label{thm-1}
Suppose that $f=u+iv\in\mathscr{PH}_{n}(\kappa)$, where $n\geq 1$.
\begin{enumerate}
\item
There exists $p_{0}=1$ such that if
$u\in \mathbf{h}^{p}(\mathbb{B}^{n},\mathbb{R}^{n})$ for some $p>p_{0}$, then $f\in\mathbf{h}^{p}(\mathbb{B}^{n},\mathbb{C}^{n})$.
Furthermore,  the constant $p_{0}=1$ is sharp.
\item
There is a positive constant $\tau_3$ such that
$$M_{p}(r,v)\leq \tau_3 M_{p}(r,u)\ \ \forall\ \ r\in[0,1),$$
for every $u\in \mathbf{h}^{p}(\mathbb{B}^{n},\mathbb{R}^{n})$ with $p>1$ and every $v$ satisfying $f=u+iv\in\mathscr{PH}_{n}(\kappa)$, where $\tau_3=\tau_3(\kappa,p,n)$.
 In particular, we have the following statements:

\begin{enumerate}
\item %For $p\in(1,2]$, $\tau_3(0,p,n)=\sqrt{n}\cot\frac{\pi}{2p^{\ast}}$, and further, $\tau_3(0,p,1)=\cot\frac{\pi}{2p^{\ast}}$
%which is the same as that in the classical analytic case, where $p^{\ast}=\max\{p,~\frac{p}{p-1}\}$.
For $p\in(1,2]$ and $\kappa=0$, if $u\in \mathbf{h}^{p}(\mathbb{B}^{n},\mathbb{R}^{n})$, then
\be\label{jkm-01}M_{p}(r,v)\leq \sqrt{n}\cot\frac{\pi}{2p^{\ast}} M_{p}(r,u)\ \ \forall\ \ r\in[0,1),\ee  where $p^{\ast}=\max\{p,~\frac{p}{p-1}\}$.
Moreover, if $n=1$, then (\ref{jkm-01}) is the same as that in the classical analytic case.
\item For $p\in(2,\infty)$ and $\kappa=0$, if $u\in \mathbf{h}^{p}(\mathbb{B}^{n},\mathbb{R}^{n})$, then
\be\label{jkm-02}M_{p}(r,v)\leq \sqrt{n}\cot\frac{\pi}{2p} M_{p}(r,u)\ \ \forall\ \ r\in[0,1).\ee
Furthermore, if $n=1$, then (\ref{jkm-02}) is also the same as that in the classical analytic case.
%For $p\in(2,\infty)$, $\tau_3(0,p,n)=\sqrt{n}\cot\frac{\pi}{2p}$, and further, $\tau_3(0,p,1)=\cot\frac{\pi}{2p}$
%which is the same as that in the classical analytic case.
\end{enumerate}\end{enumerate}
\end{Thm}

\begin{Rem}
Theorem \ref{thm-1} shows that the answer to Question \ref{Question-1.1} is affirmative for the class of $\kappa$-pluriharmonic mappings.
Also,
Theorem \ref{thm-1} is a version of Theorems E and F in $\mathbb{C}^n$ for all $n\geq 1$, which are the first two main results by Liu and Zhu \cite{L-Z}.
In particular, when $n=1$,
$(i)$ Theorem \ref{thm-1} is a unified form of Theorems E and F; $(ii)$ Theorem \ref{thm-1} implies that the assumption of ``$u\geq0$ or $u<0$" in Theorem E is redundant and the assumption of ``quasiconformality" on the mappings in Theorem F can be replaced the weaker one of ``quasiregularity".
 \end{Rem}

The rest of the paper is arranged as follows.  In Section \ref{csw-sec4}, the proof of Theorem \ref{lem-im-1} is presented,
Theorem \ref{thm-4}
is proved in Section \ref{csw-sec3}, and Section \ref{csw-sec2} is devoted to the proofs of Proposition \ref{Prop-1} and Theorem \ref{thm-1}.

\section{Riesz conjugate functions theorem for quasiregular mappings satisfying the Heinz's nonlinear
differential inequality}\label{csw-sec4}

The purpose of this section is to prove Theorem \ref{lem-im-1}. Before the proof, let us recall three useful results from \cite{Sto-2014}. The precise statements of these results need the following notions and notations.

For convenience, in this section, we assume that $\alpha>1$ denotes a constant.
For $\zeta\in\mathbf{S}^{n-1}$, let $\Gamma_{\alpha}(\zeta)$ denote the non-tangential approach region at $\zeta$, that is, $$\Gamma_{\alpha}(\zeta)=\left\{y\in\mathbf{B}^{n}:~|y-\zeta|<\alpha(1-|y|)\right\}.$$

Suppose that $\phi$ is a continuous function on $\mathbf{B}^{n}$. Then the {\it non-tangential maximal function} of $\phi$ on $\mathbf{S}^{n-1}$, denoted $M_{\alpha}[\phi]$, is defined as follows: For
$\zeta\in\mathbf{S}^{n-1}$,
$$M_{\alpha}[\phi](\zeta)=\sup\{|\phi(x)|:~x\in\Gamma_{\alpha}(\zeta)\}.$$

Suppose that $\psi\in C^{2}(\mathbf{B}^{n})$ is a non-negative subharmonic function of $\mathbf{B}^{n}$ into $\mathbb{R}$. Note that $\Delta \psi\geq 0$ (cf. \cite[Page 224]{ABR-2001}). Then it follows from $$\Delta \psi^{2}=2\left(|\nabla \psi|^{2}+\psi\Delta \psi\right)$$ that $\Delta \psi^{2}\geq 0$ on $\mathbf{B}^n$.
Let
$$
\begin{cases}
\mathcal{G}(\zeta,\psi)=\left[\int_{0}^{1}(1-r)\Delta \psi^{2}(r\zeta)dr\right]^{\frac{1}{2}};\\ $$\mathcal{S}_{\alpha}(\zeta,\psi)=\left[\int_{\Gamma_{\alpha}(\zeta)}(1-|y|)^{2-n}\Delta \psi^{2}(y)dV_{N}(y)\right]^{\frac{1}{2}}.
\end{cases}
$$

\begin{ThmI}{\rm(\cite[Corollary 3.1]{Sto-2014})}\label{Thm-A}
Suppose that $\psi\in C^{2}(\mathbf{B}^{n})$ is a non-negative subharmonic function of $\mathbf{B}^{n}$ into $\mathbb{R}$ such that $\Delta \psi^{2}$
is quasi-nearly subharmonic in $\mathbf{B}^{n}$. Then
$$
\exists\ \text{a positive constant}\ \tau_7=\tau_7(\alpha, n,p)\ \text{such that}\  \mathcal{G}(\zeta,\psi)\leq \tau_7\mathcal{S}_{\alpha}(\zeta,\psi)\ \text{in}\  \mathbf{S}^{n-1}.
$$
\end{ThmI}

\begin{ThmJ}{\rm (\cite[Theorem 4.1]{Sto-2014})}\label{Thm-B}
Suppose that $p_{0}>0$ is a constant and $\psi\in C^{2}(\mathbf{B}^{n})$ is a non-negative subharmonic function of $\mathbf{B}^{n}$ into $\mathbb{R}$ such that $\psi^{p_{0}}$
is subharmonic. If $\psi\in\mathbb{H}^{p}(\mathbf{B}^{n}, \mathbb{R})$ for some $p>p_{0}$, then
$$\int_{\mathbf{S}^{n-1}}\mathcal{S}_{\alpha}^{p}(\zeta,\psi)d\sigma(\zeta)\leq\,C_{11}\|\psi\|_{p}^{p}\ \ \text{where}\ \ C_{11}=C_{11}(\alpha,n,p).
$$
\end{ThmJ}

\begin{ThmK}\label{lem-x-1}{\rm (\cite[Lemma 2.3]{Sto-2014})}
Suppose that $\varphi\in L^{p}(\mathbf{S}^{n-1})$ for some $p\in(1,\infty)$. Then
there is
a positive constant $ C_{12}=C_{12}(\alpha,p)$ such that
$$\int_{\mathbf{S}^{n-1}}\big(M_{\alpha}[P[\varphi]](\zeta)\big)^{p}\sigma(\zeta)\leq\,C_{12}\|\varphi\|_{L^{p}(\mathbf{S}^{n-1})}^{p}.$$
\end{ThmK}

The following result concerning the non-tangential maximal functions is useful for the proof of Theorem \ref{lem-im-1}.

\begin{Lem}\label{Lem-J}
Suppose that $p_{0}\in(0,1]$ be a constant, and let $\phi:~\mathbf{B}^{n}\mapsto[0,\infty)$ be such that $\phi^{p_{0}}$ is a subharmonic function.
\begin{enumerate}
\item\label{Lem-J-1}
If $M_{\alpha}[\phi]\in L^{p}(\mathbf{S}^{n-1})$  for some $p>p_{0}$, then  $$\phi\in\mathbb{H}^{p}(\mathbf{B}^{n}, \mathbb{R})\ \ \text{with}\ \  \|\phi\|_{p}\leq\|M_{\alpha}[\phi]\|_{L^{p}(\mathbf{S}^{n-1})}.
$$
\item\label{Lem-J-2}
If $\phi\in\mathbb{H}^{p}(\mathbf{B}^{n}, \mathbb{R})$ for some $p>p_{0}$, then  $$M_{\alpha}[\phi]\in L^{p}(\mathbf{S}^{n-1})\ \ \text{with}\ \  \|M_{\alpha}[\phi]\|_{L^{p}(\mathbf{S}^{n-1})}\leq\,\tau_4\|\phi\|_{p}\ \text{where}\ \ \tau_4=\tau_4(\alpha,p,p_0).
$$
\end{enumerate}
\end{Lem}
\begin{proof}
Thanks to
$$|\phi(r\zeta)|^{p}\leq \big(M_{\alpha}[\phi](\zeta)\big)^{p}\ \ \forall\ \ (r,\zeta)\in (0,1)\times\mathbf{S}^{n-1},
$$
we see from the assumption of $M_{\alpha}[\phi]\in L^{p}(\mathbf{S}^{n-1})$ that
$$\phi\in\mathbb{H}^{p}(\mathbf{B}^{n}, \mathbb{R})\ \ \text{with}\ \
\|\phi\|_{p}\leq\|M_{\alpha}[\phi]\|_{L^{p}(\mathbf{S}^{n-1})}.
$$

To prove the statement \eqref{Lem-J-2} in the lemma, assume that $\phi\in\mathbb{H}^{p}(\mathbf{B}^{n}, \mathbb{R})$
for some $p>p_{0}$. Let $\mathfrak{F}=\phi^{p_{0}}$. Then $\mathfrak{F}\in\mathbb{H}^{p/p_{0}}(\mathbf{B}^{n}, \mathbb{R})$,
and so, the subharmonicity of $\mathfrak{F}$ ensures that
\beqq\mathfrak{F}^{\ast}(\zeta)=\lim_{\rho\rightarrow1^{-}}
\mathfrak{F}(\rho\zeta)=\lim_{\rho\rightarrow1^{-}}\left(  \phi (\rho\zeta)\right)^{p_{0}}\eeqq
exists a.e. on $\mathbf{S}^{n-1}$.
Since $\mathfrak{F}$
is a non-negative subharmonic function, we see that for $x\in\mathbf{B}^{n}$,
$$\mathfrak{F}(x)\leq\int_{\mathbf{S}^{n-1}}\mathbf{P}(x,\zeta)\mathfrak{F}^{\ast}(\zeta)d\sigma(\zeta).$$
Consequently,
 $$(\phi(x))^{p}=(\mathfrak{F}(x))^{\frac{p}{p_{0}}}
\leq\left(\int_{\mathbf{S}^{n-1}}\mathbf{P}(x,\zeta)\mathfrak{F}^{\ast}(\zeta)d\sigma(\zeta)\right)^{\frac{p}{p_{0}}}=\left(P[\mathfrak{F}^{\ast}](x)\right)^{\frac{p}{p_{0}}},$$
and thus, it follows from Theorem K that
\begin{align*}\int_{\mathbf{S}^{n-1}}\big(M_{\alpha}[\phi](\zeta)\big)^{p}\sigma(\zeta)
&\leq\int_{\mathbf{S}^{n-1}}\big(M_{\alpha}[P[\mathfrak{F}^{\ast}]](\zeta)\big)^{\frac{p}{p_{0}}}\sigma(\zeta)\\
&\leq\,C_{12}\int_{\mathbf{S}^{n-1}}\big(\mathfrak{F}^{\ast}(\zeta)\big)^{\frac{p}{p_{0}}}\sigma(\zeta)\\
&=\,C_{12}\|\phi\|_{p}^{p}\ \ \text{with}\ \  C_{12}=C_{12}(\alpha, p, p_0).
\end{align*}
Hence the lemma is proved.
\end{proof}

\subsection*{Proof of Theorem \ref{lem-im-1}}
  Without loss of generality, we assume that $j=1$ and $f_{{1}}\in \mathbb{H}^{p}(\mathbf{B}^{n}, \mathbb{R})$ for some $p>1$. We split the proof into two cases.

\noindent $\mathbf{Case~1.}$\label{4-18-1}
Suppose that $p\in(1,2]$.

For $m\in\{1,2,\cdots\}$, let $$
F_{m}(x)=\left(|f(x)|^{2}+\frac{1}{m}\right)^{\frac{1}{2}}~\mbox{and}~F_{1,m}(x)=\left(f_{1}^{2}(x)+\frac{1}{m}\right)^{\frac{1}{2}}\ \ \forall\ \  x\in\mathbf{B}^{n}.
$$

First, we get an estimate on the quantity $M_{p}^{p}(r,F_{m})$ which is stated in \eqref{eq-r06} below, where $r\in(0,1)$.
Since
\beqq
\frac{\partial^{2}F_{m}^{p}}{\partial x_{j}^{2}}=\frac{p}{2}\left(\frac{p}{2}-1\right)F_{m}^{p-4}\left(\frac{\partial|f|^{2}}{\partial x_{j}}\right)^{2}
+\frac{p}{2}F_{m}^{p-2}\frac{\partial^{2}|f|^{2}}{\partial x_{j}^{2}}\ \ \forall \ \jmath\in\{1,\cdots,n\},
\eeqq
we see that
\beqq\label{eq-c1}
\Delta F_{m}^{p}=\frac{p}{2}\left(\frac{p}{2}-1\right)F_{m}^{p-4}\left|\nabla(|f|^{2})\right|^{2}+\frac{p}{2}
F_{m}^{p-2}\Delta |f|^{2},
\eeqq
where
\be\label{xc-0.1}
\left|\nabla(|f|^{2})\right|^{2}=4\sum_{j=1}^{n}\left|\langle f, f_{x_{j}}\rangle\right|^{2}\leq4|f|^{2}\sum_{j=1}^{n}|\nabla f_{j}|^{2},
\ee
$$f_{x_{j}}=\left(\frac{\partial f_{1}}{\partial x_{j}},\ldots,\frac{\partial f_{n}}{\partial x_{j}}\right)$$
and
\be\label{xa-0.1}
\Delta |f|^{2}=2\Big(\sum_{j=1}^{n}|\nabla f_{j}|^{2}+\sum_{j=1}^{n}f_{j}\Delta f_{j}\Big).
\ee

As Theorem O   gives
$$M_{p}^{p}(r,F_{m})  =
\big(F_{m}(0)\big)^{p}+\int_{ \mathbf{B}^{n}_{r}}\Delta\left[(F_{m}(x))^{p}\right]G_{n}(x,r)\,dV_{N}(x),$$
we see from the assumption of $p\in(1,2]$ in this case that

\beq\label{eq-r01}
M_{p}^{p}(r,F_{m})
&\leq&\frac{p}{2}\int_{ \mathbf{B}^{n}_{r}}
(F_{m}(x))^{p-2}\Delta(|f(x)|^{2})G_{n}(x,r)\,dV_{N}(x)\\ \nonumber
&&+\big(F_{m}(0)\big)^{p}\\  \nonumber
&=&\big(F_{m}(0)\big)^{p}+I_{1}(r)+I_{2}(r),
\eeq where
$$
\begin{cases}
I_{1}(r)=p\int_{ \mathbf{B}^{n}_{r}}
(F_{m}(x))^{p-2}\left(\sum_{j=1}^{n}f_{j}(x)\Delta f_{j}(x)\right)G_{n}(x,r)\,dV_{N}(x);\\
I_{2}(r)=p\int_{ \mathbf{B}^{n}_{r}}
(F_{m}(x))^{p-2}\left(\sum_{j=1}^{n}|\nabla f_{j}(x)|^{2}\right)G_{n}(x,r)\,dV_{N}(x).
\end{cases}
$$

Note that $$f\in\mathscr{H}_{a}(\mathbf{B}^{n},\mathbb{R}^{n}).$$ So, this assumption gives
\beqq
0 \leq\sum_{j=1}^{n}f_{j}(x)\Delta f_{j}(x)\leq a(x)\left(\sum_{j=1}^{n}|\nabla f_{j}(x)|^{2}\right)\ \ \forall\ \ x\in\mathbf{B}^{n},
\eeqq
which implies
\beqq\label{eq-r02}
I_{1}(r)\leq\sup_{x\in\mathbf{B}^{n}}\{a(x)\}I_{2}(r).
\eeqq
Then by (\ref{eq-r01}), we have
\be\label{eq-r05}
M_{p}^{p}(r,F_{m})\leq\big(F_{m}(0)\big)^{p}+\left(1+\sup_{x\in\mathbf{B}^{n}}\{a(x)\}\right)I_{2}(r).
\ee

Since the assumption of $f$ being $K$-quasiregular ensures
\beqq
|\nabla f_{j}|\leq K|\nabla f_{k}|\ \ \forall\ \ k,j\in\{1,\ldots,n\},
\eeqq
we obtain
\be\label{eq-KK-61f}
|\nabla\, f_{j}(x)|^{2}\leq\,K^{2}|\nabla f_{1}(x)|^{2}\ \ \forall\ \ (x,j)\in\mathbf{B}^{n}\times\{2,\ldots,n\}.
\ee
Then it follows from (\ref{eq-r05}) that

\beq\label{eq-r06}
M_{p}^{p}(r,F_{m})&\leq& \big(F_{m}(0)\big)^{p}+p\left[1+(n-1)K^{2}\right]C_{a}\\ \nonumber
&&\times\int_{ \mathbf{B}^{n}_{r}}
(F_{m}(x))^{p-2}|\nabla f_{1}(x)|^{2}G_{n}(x,r)\,dV_{N}(x)\\ \nonumber
&\leq&\big(F_{m}(0)\big)^{p}+p\left[1+(n-1)K^{2}\right]C_{a}\\ \nonumber
&&\times\int_{ \mathbf{B}^{n}_{r}}
(F_{1,m}(x))^{p-2}|\nabla f_{1}(x)|^{2}G_{n}(x,r)\,dV_{N}(x),
\eeq
where $C_{a}=1+\sup_{x\in\mathbf{B}^{n}}\{a(x)\}$.
Second, based on \eqref{eq-r06}, we find a lower bound for the quantity $M_{p}^{p}(r,F_{1,m})$ in terms of $M_{p}^{p}(r,F_{m})$ as stated in \eqref{eqr-10} below.
Elementary calculations lead to
\beq\label{eqr04}
\Delta F_{1,m}^{p}&=&\frac{p}{2}\left(\frac{p}{2}-1\right)F_{1,m}^{p-4}\left|\nabla f_{1}^{2}\right|^{2}+\frac{p}{2}
F_{1,m}^{p-2}\Delta f_{1}^{2}\\ \nonumber
&=&p(p-2)F_{1,m}^{p-4}f_{1}^{2}|\nabla f_{1}|^{2}
+p F_{1,m}^{p-2}|\nabla f_{1}|^{2}\\ \nonumber
&&+p F_{1,m}^{p-2}f_{1}\Delta f_{1}.
\eeq
Then it follows from the facts of $p(p-2)\leq0$ and $f_{1}^{2}<f_{1}^{2}+\frac{1}{m}=(F_{1,m})^{2}$ that
\beq\label{eqr07}
\Delta F_{1,m}^{p}&\geq&p(p-2)F_{1,m}^{p-4}(F_{1,m})^{2}|\nabla f_{1}|^{2}
+p F_{1,m}^{p-2}|\nabla f_{1}|^{2}\\ \nonumber
&&+p F_{1,m}^{p-2}f_{1}\Delta f_{1}\\ \nonumber
&=&
 p(p-1)F_{1,m}^{p-2}|\nabla f_{1}|^{2}
+p F_{1,m}^{p-2}f_{1}\Delta f_{1},
\eeq
and thus, by Theorem O,   we have
\beqq
M_{p}^{p}(r,F_{1,m})&=&\big(F_{1,m}(0)\big)^{p}+\int_{ \mathbf{B}^{n}_{r}}\Delta\big((F_{1,m}(x))^{p}\big)G_{n}(x,r)\,dV_{N}(x)\\
&\geq&\big(F_{1,m}(0)\big)^{p}\\
&&+p(p-1)\int_{ \mathbf{B}^{n}_{r}}(F_{1,m}(x))^{p-2}|\nabla f_{1}(x)|^{2}G_{n}(x,r)\,dV_{N}(x)\\
&&+p\int_{ \mathbf{B}^{n}_{r}}(F_{1,m}(x))^{p-2}f_{1}(x)\Delta (f_{1}(x))G_{n}(x,r)\,dV_{N}(x),
\eeqq
which, together with (\ref{eq-r06}) and the assumption of $f_{1}\Delta f_{1}\geq0$ in $\mathbf{B}^{n}$, yields that
\be\label{eqr-10}
M_{p}^{p}(r,F_{1,m})\geq\big(F_{1,m}(0)\big)^{p}+\tau_5\left(M_{p}^{p}(r,F_{m})-\big(F_{m}(0)\big)^{p}\right),
\ee
where $$\tau_5=\frac{p-1}{\left[1+(n-1)K^{2}\right]\left(1+\sup_{x\in\mathbf{B}^{n}}\{a(x)\}\right)}.$$

Now, we are ready to obtain an upper bound for the quantity $M_{p}^{p}(r,f)$ in terms of $M_{p}^{p}(r,f_{1})$.
By applying the dominated convergence theorem to (\ref{eqr-10}), we have
\beqq\label{eqr-11}
M_{p}^{p}(r,f_{1})&=&
\lim_{m\rightarrow\infty}M_{p}^{p}(r,F_{1,m})\\ \nonumber
&\geq&\lim_{m\rightarrow\infty}\big(F_{1,m}(0)\big)^{p}+\tau_5\lim_{m\rightarrow\infty}\left(M_{p}^{p}(r,F_{m})-\big(F_{m}(0)\big)^{p}\right)\\ \nonumber
&=&|f_{1}(0)|^{p}+\tau_5\left(M_{p}^{p}(r,f)-|f(0)|^{p}\right),
\eeqq which gives
\beqq\label{mm-1}
M_{p}^{p}(r,f)
\leq |f(0)|^{p}+\frac{1}{\tau_5}\left(M_{p}^{p}(r,f_{1})-|f_{1}(0)|^{p}\right).
\eeqq

\noindent $\mathbf{Case~2}$\label{4-18-2}
Suppose that $p\in(2,\infty)$.

Since for $x\in\mathbf{B}^{n}$,
$$
\Delta |f(x)|^{p}=\frac{p}{2}\left(\frac{p}{2}-1\right)|f(x)|^{p-4}\left|\nabla(|f(x)|^{2})\right|^{2}+\frac{p}{2}
|f(x)|^{p-2}\Delta|f(x)|^{2},$$
we know from (\ref{xc-0.1}) and (\ref{xa-0.1}) that
\beqq
\Delta |f(x)|^{p}\leq p(p-1)|f(x)|^{p-2}\sum_{j=1}^{n}|\nabla f_{j}(x)|^{2}+p|f(x)|^{p-2}\sum_{j=1}^{n}f_{j}(x)\Delta f_{j}(x).
\eeqq
Then the assumption of $f\in\mathscr{H}_{a}(\mathbf{B}^{n},\mathbb{R}^{n})$ gives
\beq\label{eq-vchw-1}
\Delta |f(x)|^{p}&\leq&p(p-1)|f(x)|^{p-2}\sum_{j=1}^{n}|\nabla f_{j}(x)|^{2}\\ \nonumber
&&+pa(x)|f(x)|^{p-2}\sum_{j=1}^{n}|\nabla f_{j}(x)|^{2}\\ \nonumber
&\leq&p\left[(p-1)+\sup_{x\in\mathbf{B}^{n}}\{a(x)\}\right]\\ \nonumber
&&\times|f(x)|^{p-2}\sum_{j=1}^{n}|\nabla f_{j}(x)|^{2},
\eeq
and thus, it follows from Theorem O that for $r\in(0,1)$,
\beq\label{fgh-1.0}
M_{p}^{p}(r,f)&=&
|f(0)|^{p}+\int_{ \mathbf{B}^{n}_{r}}\Delta\left(|f(x)|^{p}\right)G_{n}(x,r)\,dV_{N}(x)\\ \nonumber
&\leq&|f(0)|^{p}+p\left((p-1)+\sup_{x\in\mathbf{B}^{n}}\{a(x)\}\right)I_{5}(r),
\eeq
where $$I_{5}(r)=\int_{ \mathbf{B}^{n}_{r}}
|f(x)|^{p-2}\sum_{j=1}^{n}|\nabla f_{j}(x)|^{2}G_{n}(x,r)\,dV_{N}(x).$$

Next, we are going to get an estimate on $I_{5}(r)$ which is stated in \eqref{yjh-1} below.
 By (\ref{eq-KK-61f}), we have

\beq\label{ji-1}
I_{5}(r)&=&\int_{0}^{r}n\rho^{n-1}G_{n}(\rho,r)\\  \nonumber
&&\times\left(\int_{\mathbf{S}^{n-1}}|f(\rho\zeta)|^{p-2}\sum_{j=1}^{n}|\nabla f_{j}(\rho\zeta)|^{2}d\sigma(\zeta)\right)d\rho\\ \nonumber
&\leq& \tau_6\int_{0}^{r}n\rho^{n-1}G_{n}(\rho,r)\\  \nonumber
&&\times\left(\int_{\mathbf{S}^{n-1}}|f(\rho\zeta)|^{p-2}|\nabla f_{1}(\rho\zeta)|^{2}d\sigma(\zeta)\right)d\rho\\  \nonumber
&=& \tau_6r\int_{0}^{1}n(tr)^{n-1}G_{n}(tr,r)\\  \nonumber
&&\times
\left(\int_{\mathbf{S}^{n-1}}|f(tr\zeta)|^{p-2}|\nabla f_{1}(tr\zeta)|^{2}d\sigma(\zeta)\right)dt,
\eeq where $$\tau_6=1+(n-1)K^{2}.$$

For $r\in(0,1)$, let $$  f_{r}^{\ast}
(x)=(f_{r,1}(x),\ldots,f_{r,n}(x))=(f_{1}(rx),\ldots,f_{n}(rx))=f(rx)$$ in $\mathbf{B}^{n}$.
Since when $n\geq3$,
\beqq \label{eq-vchw-3}
n(tr)^{n-1}G_{n}(tr,r)=\frac{rt(1-t)(1+t+\ldots+t^{n-3})}{n-2}\leq rt(1-t)\leq r(1-rt),
\eeqq
and when $n=2$,
\beqq \label{eq-vchw-4}
2rtG_{2}(tr,r)=rt\log\frac{1}{t}\leq r(1-rt),
\eeqq
we see from \eqref{ji-1} and Fubini's theorem that
\beqq
I_{5}(r)&\leq& \tau_6r^{2}\int_{0}^{1}(1-rt)\left(\int_{\mathbf{S}^{n-1}}|f(tr\zeta)|^{p-2}|\nabla f_{1}(tr\zeta)|^{2}d\sigma(\zeta)\right)dt\\
&=& \tau_6r^{2}\int_{\mathbf{S}^{n-1}}\left(\int_{0}^{1}(1-rt)|f(tr\zeta)|^{p-2}|\nabla f_{1}(tr\zeta)|^{2}dt\right)d\sigma(\zeta)\\
&\leq& \tau_6r^{2}\int_{\mathbf{S}^{n-1}}\left(M_{\alpha}[f_{r}^{\ast}](\zeta)\right)^{p-2}\left(\int_{0}^{1}(1-rt)|\nabla f_{1}(tr\zeta)|^{2}dt\right)d\sigma(\zeta),
\eeqq
which, together with H\"older's inequality, yields that
\beq\label{huv-0}
I_{5}(r)
&\leq & \tau_6r\int_{\mathbf{S}^{n-1}}\left(M_{\alpha}[f_{r}^{\ast}](\zeta)\right)^{p-2}\mathscr{F}_{r}(\zeta)d\sigma(\zeta)\\  \nonumber
&\leq& \tau_6 r\left(\int_{\mathbf{S}^{n-1}}\left(M_{\alpha}[f_{r}^{\ast}](\zeta)\right)^{p}d\sigma(\zeta)\right)^{\frac{p-2}{p}}
\left(\int_{\mathbf{S}^{n-1}}\left(\mathscr{F}_{r}(\zeta)\right)^{\frac{p}{2}}d\sigma(\zeta)\right)^{\frac{2}{p}},
\eeq where $$\mathscr{F}_{r}(\zeta)=r\int_{0}^{1}(1-rt)|\nabla f_{1}(tr\zeta)|^{2}dt=\int_{0}^{r}(1-\rho)|\nabla  f_{1}(\rho\zeta)|^{2}d\rho.$$

Note that
$$\Delta F_{1,m}\geq(F_{1,m})^{-1}f_{1}\Delta f_{1}\geq0\ \ \forall\ \ m\in \{1,2,\ldots\}.$$
This ensures that $F_{1,m}$ is subharmonic in $\mathbf{B}^{n}$. Then for all $w\in\mathbf{B}^{n}$ and $\varrho\in[0,1-|w|)$ there holds

$$F_{1,m}(w)\leq\int_{\mathbf{S}^{n-1}}F_{1,m}(w+\varrho\zeta)d\sigma(\zeta),$$
which guarantees
\beqq
|f_{1}(w)|&=&\lim_{m\rightarrow\infty}F_{1,m}(w)\leq\lim_{m\rightarrow\infty}\int_{\mathbf{S}^{n-1}}F_{1,m}(w+\varrho\zeta)d\sigma(\zeta)\\
&\leq&\lim_{m\rightarrow\infty}\int_{\mathbf{S}^{n-1}}\left(|f_{1}(w+\varrho\zeta)|+\frac{1}{\sqrt{m}}\right)d\sigma(\zeta)\\
&=&\int_{\mathbf{S}^{n-1}}|f_{1}(w+\varrho\zeta)|d\sigma(\zeta).
\eeqq
Consequently, $|f_{1}|$ is subharmonic in $\mathbf{B}^{n}$, and thus, the assumption of $f_{{1}}\Delta f_{{1}}\geq 0$ gives
\beqq\label{dfg-1}\Delta f_{1}^{2}=2\left(|\nabla f_{1}|^{2}+f_{1}\Delta f_{1}\right)\geq2|\nabla f_{1}|^{2}. \eeqq
It follows that
\beq\label{huv-0-1}
\nonumber
\left(\int_{\mathbf{S}^{n-1}}\left(\mathscr{F}_{r}(\zeta)\right)^{\frac{p}{2}}d\sigma(\zeta)\right)^{\frac{2}{p}}&\leq&
\left(\int_{\mathbf{S}^{n-1}} \left(\mathscr{F}_{1}(\zeta)\right)^{\frac{p}{2}}  d\sigma(\zeta)\right)^{\frac{2}{p}}\\
&  \leq   & \frac{1}{2}  \left(\int_{\mathbf{S}^{n-1}}\big(\mathcal{G}(\zeta,|f_{1}|)\big)^{p}d\sigma(\zeta)\right)^{\frac{2}{p}}.
\eeq

Since Theorem I and the assumption of $\Delta f_{{1}}^{2}$ being quasi-nearly subharmonic ensure that
$$\mathcal{G}(\zeta,|f_{1}|)\leq \tau_7\mathcal{S}_{2}(\zeta,|f_{1}|),$$ where $\tau_7=C_{10}(2,n,p)$ and $C_{10}$ is from Theorem I,
and since by letting $p_0=1$, the fact of $|f_1|$ being subharmonic, the assumption of $f_{{1}}\in \mathbb{H}^{p}(\mathbf{B}^{n}, \mathbb{R})$ and Theorem J show that
$$\int_{\mathbf{S}^{n-1}}\mathcal{S}_{2}^{p}(\zeta,|f_1|)d\sigma(\zeta)\leq\,\tau_8\|f_1\|_{p}^{p},$$ where $\tau_8=C_{11}(2,n,p)$ and $C_{11}$ is from Theorem J, we see from (\ref{dfg-1}) that
\beq\label{huv-0-1}
\left(\int_{\mathbf{S}^{n-1}}\left(\mathscr{F}_{r}(\zeta)\right)^{\frac{p}{2}}d\sigma(\zeta)\right)^{\frac{2}{p}} \leq \frac{1}{2}\tau_{7}^{2}\tau_8\|f_{1}\|_{p}^{2}.
\eeq

On the other hand, for $m\in\{1,2,\ldots,\}$, let $$\mathbf{F}_{r,m}(x)=\left(|f_{r}^{\ast}(x)|^{2}+\frac{1}{m}\right)^{\frac{1}{2}}$$ in $\mathbf{B}^{n}$.
Then  we have

\beqq
\Delta \mathbf{F}_{r,m}&=&-\frac{1}{4}\mathbf{F}_{r,m}^{-3}\left|\nabla(|f_{r}^{\ast}|^{2})\right|^{2}+\frac{1}{2}
\mathbf{F}_{r,m}^{-1}\Delta |f_{r}^{\ast}|^{2}\\
&\geq&-\mathbf{F}_{r,m}^{-3}|f_{r}^{\ast}|^{2}\sum_{j=1}^{n}|\nabla f_{r,j}|^{2}+\mathbf{F}_{r,m}^{-1}\sum_{j=1}^{n}|\nabla f_{r,j}|^{2}+\mathbf{F}_{r,m}^{-1}\sum_{j=1}^{n}f_{r,j}\Delta f_{r,j}\\
&\geq&\mathbf{F}_{r,m}^{-1}\sum_{j=1}^{n}f_{r,j}\Delta f_{r,j}\geq0,
\eeqq
which implies that $\mathbf{F}_{r,m}$ is subharmonic. Then, for all $w\in\mathbf{B}^{n}$ and $\rho\in[0,1-|w|)$,
we obtain
\beqq
|f_{r}^{\ast}(w)|&=&\lim_{m\rightarrow\infty}\mathbf{F}_{r,m}(w)\leq\lim_{m\rightarrow\infty}\int_{\mathbf{S}^{n-1}}\mathbf{F}_{r,m}(w+\rho\zeta)d\sigma(\zeta)\\
&\leq&\lim_{m\rightarrow\infty}\int_{\mathbf{S}^{n-1}}\left(|f_{r}^{\ast}(w+\rho\zeta)|+\frac{1}{\sqrt{m}}\right)d\sigma(\zeta)\\
&=&\int_{\mathbf{S}^{n-1}}|f_{r}^{\ast}(w+\rho\zeta)|d\sigma(\zeta).
\eeqq
This shows that $|f_{r}^{\ast}|$ is also subharmonic in $\mathbf{B}^{n}$. By taking $p_0=1$ and $\alpha=2$, it follows from Lemma \ref{Lem-J}\eqref{Lem-J-2} that
\beq\label{huv-1}
\left(\int_{\mathbf{S}^{n-1}}\left(M_{\alpha}[f_{r}^{\ast}](\zeta)\right)^{p}\sigma(\zeta)\right)^{\frac{p-2}{p}}&\leq&\,\tau_9\|f_{r}^{\ast}\|_{L^{p}(\mathbf{S}^{n-1})}^{p-2}\\
\nonumber
=\tau_9 M_{p}^{p-2}(r,f),
\eeq where $\tau_9=\tau_4(2, p,1)$ and $\tau_4$ is from Lemma \ref{Lem-J}.
By applying (\ref{huv-0-1}) and (\ref{huv-1}) to (\ref{huv-0}), we get
\be\label{yjh-1}
I_{5}(r)
\leq\, \frac{1}{2}\tau_6\tau_{7}^{2}\tau_8\tau_9   M_{p}^{p-2}(r,f)\|f_{1}\|_{p}^{2}.
\ee

Now, by (\ref{fgh-1.0}), we see from (\ref{yjh-1}) that
\beq\label{fgh-1.2}
M_{p}^{p}(r,f)\leq |f(0)|^{p}+\tau_{10} M_{p}^{p-2}(r,f)\|f_{1}\|_{p}^{2},
\eeq
where $$\tau_{10}=\frac{1}{2}p\Big((p-1)+\sup_{x\in\mathbf{B}^{n}}\{a(x)\}\Big)\tau_6\tau_{7}^{2}\tau_8\tau_9.$$ Obviously, $$\tau_{10}=\tau_{10}(n, p, K).$$

Since  $f\in\mathscr{H}_{a}(\mathbf{B}^{n},\mathbb{R}^{n})$ and
\beqq
r^{n-1}\frac{d}{dr}M_{p}^{p}(r,f)
&=&\frac{p}{2n}\left(\frac{p}{2}-1\right)\int_{ \mathbf{B}^{n}_{r}}|f(x)|^{p-4}\left|\nabla(|f(x)|^{2})\right|^{2}dV_{N}(x)\\
&&+\frac{p}{2n}
\int_{ \mathbf{B}^{n}_{r}}|f(x)|^{p-2}\Delta|f(x)|^{2}dV_{N}(x)\\
&\geq&0,
\eeqq
we see that for $p>2$,  $M_{p}(r,f)$ is also increasing on $r\in(0,1)$. Then for $r\in[0,1)$,
$$|f(0)|\leq M_{p}(r,f),$$ which, together with
  (\ref{fgh-1.2}), gives that

  \beqq
  M_{p}^{p}(r,f)&\leq& |f(0)|^{p}+\tau_{10}M_{p}^{p-2}(r,f)\|f_{1}\|_{p}^{2}\\
  &\leq&|f(0)|^{2}M_{p}^{p-2}(r,f)+\tau_{10}M_{p}^{p-2}(r,f)\|f_{1}\|_{p}^{2}.
  \eeqq
  Consequently,
$$  M_{p}^{2}(r,f) \leq|f(0)|^{2}+\tau_{10}\|f_{1}\|_{p}^{2},
$$ which implies that
\be\label{mm-2}\|f\|_{p}^{2} \leq|f(0)|^{2}+\tau_{10}\|f_{1}\|_{p}^{2}\ee

Next, we prove the sharpness part. Let us consider the following example.
For any fixed $K\in[1,\infty)$, let $$f(x_{1},x_{2})=\left(f_{1}(x_{1},x_{2}), f_{2}(x_{1},x_{2})\right)\ \ \forall\ \ (x_{1},x_{2})\in\mathbf{B}^{2},$$
where
$$
\begin{cases} f_{1}(x_{1},x_{2})=\frac{2K(1-x_{1}^{2}-x_{2}^{2})}{(K+1)[(1-x_{1})^{2}+x_{2}^{2}]};\\
f_{2}(x_{1},x_{2})=
\frac{4x_{2}}{(K+1)[(1-x_{1})^{2}+x_{2}^{2}]}.
\end{cases}
$$
Then by elementary calculations, we have the following:
\begin{enumerate}
	\item[$(i)$]
	$f$ is a $K$-quasiconformal mapping.
	\item[$(ii)$]
	For each $j\in\{1,2\}$, $\Delta f_{j}=0$.
	\item[$(iii)$]
	For each $j\in\{1,2\}$, $\Delta f_{j}^{2}=2|\nabla f_{j}|^{2}$.
	\item[$(iv)$]
	 $f_{1}\in\mathbb{H}^{1}(\mathbb{D}, \mathbb{R})$ for any fixed $K\in[1,\infty)$ and
$f\notin\mathbb{H}^{1}(\mathbb{D}, \mathbb{R}^{2})$.
\end{enumerate}

Obviously, it follows from $(ii)$ that the assumption of $(1)$ in the theorem is satisfied and $f\in\mathscr{H}_{a}(\mathbf{B}^{2},\mathbb{R}^{2})$ for any positive and bounded function $a$. By $(iii)$, we see that $\Delta f_{j}^{2}$ is (quasi-nearly) subharmonic. These illustrate that $f$ satisfies all assumptions in the theorerm for the case when $n=2$. Then the sharpness of $p_{0}(=1)$ follows from $(iv)$.
The proof of this theorem is complete.
\qed

\section{Riesz conjugate functions theorem for invariant harmonic quasiregular mappings}\label{csw-sec3}
The purpose of this section is to prove Theorem \ref{thm-4}.
Before the proof, let us  recall two useful known results.

%The following result is  useful for the discussions in this section.

\begin{ThmL}{\rm (\cite[Lemma 2.5]{Pav}~{\rm or}~\cite[Theorem  4.1.1]{Sto})}\label{Gre-1.1}
Suppose that $\Psi:~\mathbf{B}^{n}\rightarrow\mathbb{R}$ is twice continuously differentiable. Then for $r\in(0,1)$,
\be
\int_{\mathbf{S}^{n-1}}\Psi(r\zeta)d\sigma(\zeta)=\Psi(0)+\int_{\mathbf{B}^{n}_r}g(|x|,r)\Delta_{h}\Psi(x)dV_{h}(x),
\ee
where $$dV_{h}(x)=(1-|x|^{2})^{-n}dV_{N}(x)\;\;\mbox{and}\;\; g(|x|,r)=\frac{1}{n}\int_{|x|}^{r}\frac{(1-s^{2})^{n-2}}{s^{n-1}}ds.$$
\end{ThmL}

For $p\in(0,\infty )$, we use $L^{p}(\mathbf{S}^{n-1})$ to denote the
set of all measurable functions $F$ of $\mathbf{S}^{n-1}$ into
$\mathbb{R}$ which satisfy

$$\|F\|_{L^{p}(\mathbf{S}^{n-1})}=\left(\int_{\mathbf{S}^{n-1}}|F(\zeta)|^{p}d\sigma(\zeta)\right)^{\frac{1}{p}}<\infty.$$

Suppose that $f$ is an  invariant harmonic function of $\mathbf{B}^{n}$ into $\mathbb{R}$.
The {\it $\widetilde{g}_{f}$-function} of $f$ is defined as follows: For $\zeta\in \mathbf{S}^{n-1}$,
$$\widetilde{g}_{f}(\zeta)=\left(\int_{0}^{1}\frac{|\nabla^{h} f(r\zeta)|^{2}}{1-r}dr\right)^{\frac{1}{2}},$$
where $\nabla^{h}f=(1-|x|^{2})\nabla f$ stands for the invariant
real gradient (see \cite{L-Z,Sto,Sto-2012} for the precise definition and properties of this gradient).

The following result concerns the relations between an invariant harmonic function $f$ and its $\widetilde{g}_{f}$.

 \begin{Lem}\label{Lem-chxw}
Suppose that $f$ is an invariant harmonic function of $\mathbf{B}^{n}$ into $\mathbb{R}$ and $p\in[2,\infty)$.
%If $\widetilde{g}_{f}\in\,L^{p}(\mathbf{S}^{n-1})$, then there exist a positive constant $C_9=C_9(n,p)$ such that
%$$\|f\|_{p}
%\leq|f(0)|+C_9\|\widetilde{g}_{f}\|_{L^{p}(\mathbf{S}^{n-1})}.$$

\begin{enumerate}
\item[{\rm ($b_1$)}]\quad  If $f\in\mathbb{H}^{p}(\mathbf{B}^{n}, \mathbb{R})$, then
$$
\exists\ \text{a positive constant}\ C_6=C_6(n,p)\  \text{such that}\
\|\widetilde{g}_{f}\|_{L^{p}(\mathbf{S}^{n-1})}
\leq\,C_6\|f\|_{p}.
$$
\item[{\rm ($b_2$)}]\quad If $\widetilde{g}_{f}\in\,L^{p}(\mathbf{S}^{n-1})$, then
$$
\exists\ \text{a positive constant}\ C_7=C_7(n,p)\ \text{such that}\
\|f\|_{p}\leq|f(0)|+C_7\|\widetilde{g}_{f}\|_{L^{p}(\mathbf{S}^{n-1})}.
$$
\end{enumerate}
\end{Lem}
\begin{proof}
Since the statement ($b_{1}$) in the lemma easily follows from \cite[Theorem 4.2]{Kw}, we only need to prove ($b_{2}$).
By H\"older's inequality and the assumption of $\widetilde{g}_{f}\in\,L^{p}(\mathbf{S}^{n-1})$, we have
\beq\label{eq-pl-1}
\int_{\mathbf{S}^{n-1}}\left(\widetilde{g}_{f}(\zeta)\right)^{2}d\sigma(\zeta)
&\leq&\left[\int_{\mathbf{S}^{n-1}}\left(\widetilde{g}_{f}(\zeta)\right)^{p}d\sigma(\zeta)\right]^{\frac{2}{p}}
\left(\int_{\mathbf{S}^{n-1}}1^{q}d\sigma(\zeta)\right)^{\frac{1}{q}}\\ \nonumber
&=&\|\widetilde{g}_{f}\|_{L^{p}(\mathbf{S}^{n-1})}^{2}<\infty.
\eeq

Since
\be\label{eq-pl-2}
g_{n}(\rho,1)\leq\frac{2^{n-2}}{n\rho^{n-1}}\int_{\rho}^{1}(1-s)^{n-2}ds=\frac{2^{n-2}}{n(n-1)}\frac{(1-\rho)^{n-1}}{\rho^{n-1}},
\ee
we know from (\ref{eq-pl-1}), Theorem L and the invariant subharmonicity of $|f|^{\gamma}$  $(\gamma>1)$ that
\beqq
\lim_{r\rightarrow1^{-}}M_{2}^{2}(r,f)
&=&f^{2}(0)+2\int_{\mathbf{B}^{n}}g(|x|,1)|\nabla^{h} f(x)|^{2}dV_{h}(x)\\
&=&f^{2}(0)+2\int_{\mathbf{S}^{n-1}}\int_{0}^{1}\frac{n\rho^{n-1}g(\rho,1)}{(1-\rho^{2})^{n}}|\nabla^{h} f(\rho\zeta)|^{2}d\rho\,d\sigma(\zeta)\\
&\leq&f^{2}(0)+\frac{2^{n-1}}{n-1}\int_{\mathbf{S}^{n-1}}\left(\widetilde{g}_{f}(\zeta)\right)^{2}d\sigma(\zeta)\\
&<&\infty.
\eeqq
%&=&f^{2}(0)+\lim_{r\rightarrow1^{-}}\int_{\mathbf{B}^{n}_r}g(|x|,r)\Delta_{h}f^{2}(x)dV_{h}(x)\\
Consequently, $f\in\mathbb{H}^{2}(\mathbf{B}^{n}, \mathbb{R})$ which
implies that the radial limits
$$f^{\ast}(\zeta)=\lim_{r\rightarrow1^{-}}f(r\zeta)$$ exist for almost every $\zeta\in\mathbf{S}^{n-1}$  (cf. \cite[p.96]{Sto}).

Let $q$ be the conjugate exponent of $p$. Then

$$\|f\|_{p}=\sup\left\{\left|\int_{\mathbf{S}^{n-1}}f^{\ast}(\zeta)\mathbf{g}(\zeta)d\sigma(\zeta)\right|:~\mathbf{g}\in\,L^{q}(\mathbf{S}^{n-1}), ~\|\mathbf{g}\|_{L^{q}(\mathbf{S}^{n-1})}\leq1 \right\}.$$

Since $C(\mathbf{S}^{n-1})$ is dense in $L^{q}(\mathbf{S}^{n-1})$, the above supremum can be taken over $\mathbf{g}\in C(\mathbf{S}^{n-1})$.
In the following, we will show that $\|f\|_{p}<\infty$.
 Elementary calculations show that
\beq\label{chwm-7}
\Delta_{h}\big(fP_{h}[\mathbf{g}]\big)&=&f\Delta_{h}P_{h}[\mathbf{g}]+P_{h}[\mathbf{g}]\Delta_{h}f+2\langle\nabla^{h}f,\nabla^{h}P_{h}[\mathbf{g}]\rangle\\ \nonumber
&=&2\langle\nabla^{h}f,\nabla^{h}P_{h}[\mathbf{g}]\rangle,
\eeq
and thus, Theorem L gives
\beq\label{xchw-1-01}\nonumber
\left|\int_{\mathbf{S}^{n-1}}f(r\zeta)P_{h}[\mathbf{g}](r\zeta)d\sigma(\zeta)\right|&\leq&
\left|\int_{\mathbf{B}^{n}_{r}}g_{n}(|x|,r)\Delta_{h}\big(f(x)P_{h}[\mathbf{g}](x)\big)dV_{h}(x)\right|\\ \nonumber
&&+\left|f(0)P_{h}[\mathbf{g}](0)\right|\\ \nonumber
&=&
2\left|\int_{\mathbf{B}^{n}_{r}}g_{n}(|x|,r)\langle\nabla^{h}f(x),\nabla^{h}P_{h}[\mathbf{g}](x)\rangle\,dV_{h}(x)\right|\\ \nonumber
&&+\left|f(0)P_{h}[\mathbf{g}](0)\right|
\\ \nonumber
&\leq&\left|f(0)\right|+2\mathscr{Y}(r),
\eeq
where
$$\mathscr{Y}(r)=\int_{\mathbf{B}^{n}_{r}}g_{n}(|x|,r)\left|\nabla^{h}f_{j}(x)\right|\left|\nabla^{h}P_{h}[\mathbf{g}](x)\right|dV_{h}(x).$$

Since (\ref{eq-pl-2}) and H\"older's inequality lead to
\beq\label{chwm-8}
\mathscr{Y}(r)&\leq&\int_{\mathbf{B}^{n}}g_{n}(|x|,1)\left|\nabla^{h}f(x)\right|\left|\nabla^{h}P_{h}[\mathbf{g}](x)\right|dV_{h}(x)\\ \nonumber
&\leq&\frac{2^{n-2}}{n-1}\int_{\mathbf{S}^{n-1}}
\left(\int_{0}^{1}\frac{\left|\nabla^{h}f(\rho\zeta)\right|\left|\nabla^{h}P_{h}[\mathbf{g}](\rho\zeta)\right|}{1-\rho}d\rho\right)d\sigma(\zeta)\\ \nonumber
&\leq&\frac{2^{n-2}}{n-1}\int_{\mathbf{S}^{n-1}}
\left(\int_{0}^{1}\frac{\left|\nabla^{h}f(\rho\zeta)\right|\left|\nabla^{h}P_{h}[\mathbf{g}](\rho\zeta)\right|}{1-\rho}d\rho\right)d\sigma(\zeta)\\ \nonumber
&\leq&\frac{2^{n-2}}{n-1}\|\widetilde{g}_{f}\|_{L^{p}(\mathbf{S}^{n-1})}\|\widetilde{g}_{P_{h}[\mathbf{g}]}\|_{L^{q}(\mathbf{S}^{n-1})},
\eeq
 we have from ($b_{1}$) and \cite[Theorem 7.1.1]{Sto} that there is a positive constant $C(n,q)$
such that

\be\label{chwm-10}
\|\widetilde{g}_{P_{h}[\mathbf{g}]}\|_{L^{p}(\mathbf{S}^{n-1})}\leq\,C(n,q)\|P_{h}[\mathbf{g}]\|_{q}  =C(n,q)\|\mathbf{g}\|_{L^{q}(\mathbf{S}^{n-1})}
  \leq\,C(n,q).
\ee

From (\ref{chwm-8}) and (\ref{chwm-10}), we obtain
\beqq\label{chwm-11}
\mathscr{Y}(r)\leq\,\frac{2^{n-2}}{n-1}C(n,q)\|\widetilde{g}_{f}\|_{L^{p}(\mathbf{S}^{n-1})},
\eeqq
which, together with
 (\ref{xchw-1-01}),
 gives

\beqq
\|f\|_{p}&\leq& \sup_{r\in(0,1)}\sup_{\mathbf{g}\in\,L^{q}(\mathbf{S}^{n}), \|\mathbf{g}\|_{L^{q}(\mathbf{S}^{n-1})}\leq1}\left|\int_{\mathbf{S}^{n-1}}f(r\zeta)P_{h}[\mathbf{g}](r\zeta)d\sigma(\zeta)\right|\\
&\leq& \left|f(0)\right|
+\frac{2^{n-1}}{n-1}C(n,q)\|\widetilde{g}_{f}\|_{L^{p}(\mathbf{S}^{n-1})}.
\eeqq
The proof of this lemma is complete.
\end{proof}

\subsection*{Proof of Theorem \ref{thm-4}}
Clearly, Theorem \ref{thm-4} reduces to Theorem \ref{thm-1} when $n=2$. A detailed proof of Theorem \ref{thm-1} will be presented in Section \ref{csw-sec2}.
Hence we only need to consider the case when $n\geq 3$.
Without loss of generality, we assume that $f_{1}\in\mathbb{H}^{p}(\mathbf{B}^{n}, \mathbb{R})$.
We separate the proof into two cases.

\noindent $\mathbf{Case~1.}$\label{chwk-1}
Suppose $p\in(1,2]$.

Let us introduce two sequences of functions which are as follows:
For $m\in\{1,2,\ldots\}$, define $$\mathcal{F}_{m}(x)=\left(|f(x)|^{2}+\frac{1}{m}\right)^{\frac{1}{2}}~\mbox{and}~\mathcal{F}_{1,m}(x)=\left(|f_{1}(x)|^{2}+\frac{1}{m}\right)^{\frac{1}{2}}$$ in $\mathbf{B}^{n}$.

In the following, we get an upper bound for $\Delta_{h}(\mathcal{F}_{m}^{p})$ in terms of $\Delta_{h}(\mathcal{F}_{1,m}^{p})$ as stated in \eqref{10-28-1}.
On the one hand, since
\beqq
\nabla(\mathcal{F}_{m}^{p})=\frac{p}{2}\left(|f|^{2}+\frac{1}{m}\right)^{\frac{p}{2}-1}\nabla\left(|f|^{2}\right)
\eeqq
and
\beqq
\Delta(\mathcal{F}_{m}^{p})&=&\frac{p}{2}\left(\frac{p}{2}-1\right)\left(|f|^{2}+\frac{1}{m}\right)^{\frac{p}{2}-2}\left|\nabla\left(|f|^{2}\right)\right|^{2}\\
&&+\frac{p}{2}
\left(|f|^{2}+\frac{1}{m}\right)^{\frac{p}{2}-1}\Delta\left(|f|^{2}\right),
\eeqq
we see that
\beq\label{eq-ry-1}
\Delta_{h}(\mathcal{F}_{m}^{p})&=&(1-|x|^{2})^{2}\Delta(\mathcal{F}_{m}^{p})+2(n-2)(1-|x|^{2})\langle \nabla(\mathcal{F}_{m}^{p}),x\rangle\\ \nonumber
&=&\frac{p}{2}\left(\frac{p}{2}-1\right)\left(|f|^{2}+\frac{1}{m}\right)^{\frac{p}{2}-2}(1-|x|^{2})^{2}\left|\nabla\left(|f|^{2}\right)\right|^{2}\\ \nonumber
&&+\frac{{p}{2}^{-1}\Delta_{h}\big(|f|^{2}\big)}{\left(|f|^{2}+\frac{1}{m}\right)^{1-\frac{p}{2}}}.
\eeq

Note that
\beq\label{eq-ry-2}
\Delta_{h}\left(|f|^{2}\right)&=&\Delta_{h}\left(\sum_{j=1}^{n}f_{j}^{2}\right)\\ \nonumber
&=&2\sum_{j=1}^{n}\left[(1-|x|^{2})^{2}|\nabla f_{j}|^{2}+f_{j}\Delta_{h}f_{j}\right]\\ \nonumber
&=&2\sum_{j=1}^{n}(1-|x|^{2})^{2}|\nabla f_{j}|^{2}.
\eeq

Since the assumption of $f$ being $K$-quasiregular in $\mathbf{B}^{n}$ ensures that for all $k,j\in\{1,\ldots,n\}$,
\beqq
|\nabla f_{j}|\leq K|\nabla f_{k}|,
\eeqq
as a consequence, we know that for each $j\in\{2,\ldots,n\}$,
\be\label{KK-1}
|\nabla f_{j}|\leq K|\nabla f_{1}|.
\ee

Since $p\in(1,2]$ (as assumed in this case), we infer from (\ref{eq-ry-1}), (\ref{eq-ry-2}) and (\ref{KK-1}) that

\beq\label{eq-KK-2}
\Delta_{h}(\mathcal{F}_{m}^{p})&\leq& p(1-|x|^{2})^{2}\left(|f|^{2}+\frac{1}{m}\right)^{\frac{p}{2}-1}\sum_{j=1}^{n}|\nabla f_{j}|^{2}\\ \nonumber
&\leq&p[1+(n-1)K^{2}](1-|x|^{2})^{2}\left(|f|^{2}+\frac{1}{m}\right)^{\frac{p}{2}-1}|\nabla f_{1}|^{2}.
\eeq

On the other hand, elementary computations ensure that
\beqq\label{eq-KK-3}
\Delta_{h}(\mathcal{F}_{1,m}^{p})
&=&\frac{p}{2}\left(\frac{p}{2}-1\right)\left(f_{1}^{2}+\frac{1}{m}\right)^{\frac{p}{2}-2}(1-|x|^{2})^{2}\left|\nabla\left(f_{1}^{2}\right)\right|^{2}\\ \nonumber
&&+\frac{p}{2}\left(f_{1}^{2}+\frac{1}{m}\right)^{\frac{p}{2}-1}\Delta_{h}\left(f_{1}^{2}\right)\\ \nonumber
&=&p(p-2)(1-|x|^{2})^{2}\left(f_{1}^{2}+\frac{1}{m}\right)^{\frac{p}{2}-2}f_{1}^{2}|\nabla f_{1}|^{2}\\ \nonumber
&&+p(1-|x|^{2})^{2}|\nabla f_{1}|^{2}\left(f_{1}^{2}+\frac{1}{m}\right)^{\frac{p}{2}-1}\\ \nonumber
&\geq&p(p-1)(1-|x|^{2})^{2}|\nabla f_{1}|^{2}\left(f_{1}^{2}+\frac{1}{m}\right)^{\frac{p}{2}-1}.
\eeqq
Then we obtain from (\ref{eq-KK-2}) that
\beq\label{10-28-1}
\Delta_{h}(\mathcal{F}_{1,m}^{p})\geq\frac{p-1}{1+(n-1)K^{2}}\Delta_{h}(\mathcal{F}_{m}^{p}).
\eeq

Now, we ready to finish the proof. As Theorem L leads to
\beq\label{eq-KK-4}\nonumber
\int_{\mathbf{S}^{n-1}}(\mathcal{F}_{1,m}(r\zeta))^{p}d\sigma(\zeta)&=&(\mathcal{F}_{1,m}(0))^{p}+
\int_{\mathbf{B}^{n}_r}g(|x|,r)\Delta_{h}(\mathcal{F}_{1,m}(x))^{p}dV_{h}(x)\\ \nonumber
&\geq&\frac{p-1}{1+(n-1)K^{2}}\int_{\mathbf{B}^{n}_r}g(|x|,r)\Delta_{h}(\mathcal{F}_{m}(x))^{p}dV_{h}(x)\\ \nonumber
&&\frac{p-1}{1+(n-1)K^{2}}(\mathcal{F}_{m}(0))^{p}+(\mathcal{F}_{1,m}(0))^{p}\\ \nonumber
&&-\frac{p-1}{1+(n-1)K^{2}}(\mathcal{F}_{m}(0))^{p}\\
&=&\frac{p-1}{1+(n-1)K^{2}}\int_{\mathbf{S}^{n-1}}(\mathcal{F}_{m}(r\zeta))^{p}d\sigma(\zeta)\\  \nonumber
&&+(\mathcal{F}_{1,m}(0))^{p}-\frac{p-1}{1+(n-1)K^{2}}(\mathcal{F}_{m}(0))^{p},
\eeq
it follows from the assumption of $f_{1}\in\mathbb{H}^{p}(\mathbf{B}^{n}, \mathbb{R})$, together with (\ref{eq-KK-4}) and the
Lebesgue's control-convergent theorem, that
\beqq
\infty&>&\lim_{m\rightarrow\infty}\int_{\mathbf{S}^{n-1}}(\mathcal{F}_{1,m}(r\zeta))^{p}d\sigma(\zeta)\\
&\geq&\frac{p-1}{1+(n-1)K^{2}}\int_{\mathbf{S}^{n-1}}\lim_{m\rightarrow\infty}(\mathcal{F}_{m}(r\zeta))^{p}d\sigma(\zeta)\\ \nonumber
&&+\lim_{m\rightarrow\infty}(\mathcal{F}_{1,m}(0))^{p}-\frac{p-1}{1+(n-1)K^{2}}\lim_{m\rightarrow\infty}(\mathcal{F}_{m}(0))^{p}.
\eeqq
This gives
\beqq
\|f\|_{p}^{p}\leq|f(0)|^{p}+\frac{1+(n-1)K^{2}}{p-1}\left(\|f_{1}\|_{p}^{p}-|f_{1}(0)|^{p}\right),
\eeqq which is what we need.

\noindent $\mathbf{Case~2.}$
Suppose $p>2$.

Since $f_{1}\in\mathbb{H}^{p}(\mathbf{B}^{n}, \mathbb{R})$, by  Lemma \ref{Lem-chxw} ($b_{1}$), we see that there is a positive
constant $C_8=C_8(n,p)$ such that
\beq\label{eq-KK-53}
\left(\int_{\mathbf{S}^{n-1}}\left(\int_{0}^{1}\frac{|\nabla^{h}f_{1}(r\zeta)|^{2}}{1-r}dr\right)^{\frac{p}{2}}d\sigma(\zeta)\right)^{\frac{1}{p}}
&=&\|\widetilde{g}_{f_{1}}\|_{L^{p}(\mathbf{S}^{n-1})}\\ \nonumber
&\leq&\,C_8\|f_{1}\|_{p}.
\eeq

By the assumption of $f$ being $K$-quasiregular in $\mathbf{B}^{n}$, we see
\beqq
|\nabla f_{j}|\leq K|\nabla f_{k}|\ \ \forall\ \ k,j\in\{1,\ldots,n\}.
\eeqq
Consequently, we have that for each $j\in\{2,\ldots,n\}$,
\beqq\label{eq-KK-54}
|\nabla^{h} f_{j}(x)|^{2}&=&(1-|x|^{2})^{2}|\nabla f_{j}(x)|^{2}\leq\,K^{2}(1-|x|^{2})^{2}|\nabla f_{1}(x)|^{2}\\
&=&K^{2}|\nabla^{h} f_{1}(x)|^{2}\ \forall\ x\in\mathbf{B}^n,
\eeqq
and thus, (\ref{eq-KK-53}) gives
 \be\label{4-30-1}
 \|\widetilde{g}_{f_{j}}\|_{L^{p}(\mathbf{S}^{n-1})}\leq K
\|\widetilde{g}_{f_{1}}\|_{L^{p}(\mathbf{S}^{n-1})} \leq C_8K\|f_{1}\|_{p}.
 \ee

Since  Lemma \ref{Lem-chxw} ($b_{2}$) ensures that there exists a positive constant
$C_9=C_9(n,p)$ satisfying
\beqq\label{eq-KK-56}
\|f_{j}\|_{p}-|f_{j}(0)|\leq\,C_9\|\widetilde{g}_{f_{j}}\|_{L^{p}(\mathbf{S}^{n-1})},
\eeqq
we get
\be\label{eq-KK-56}
\|f_{j}\|_{p}-|f_{j}(0)|\leq C_8C_9K\|f_{1}\|_{p}.
\ee

As the Minkowski inequality implies
\beqq
 M_{p}(r,f)&=&
 \left(\int_{\mathbf{S}^{n-1}}\left(\sum_{j=1}^{n}|f_{j}(r\zeta)|^{2}\right)^{\frac{p}{2}}d\sigma(\zeta)\right)^{\frac{1}{p}}\leq \left(\sum_{j=1}^{n}M_{p}^{2}(r,f_{j})\right)^{\frac{1}{2}}\\
 &\leq&\sum_{j=1}^{n}M_{p}(r,f_{j}),
\eeqq
we infer from the assumption of $f_{1}\in\mathbb{H}^{p}(\mathbf{B}^{n}, \mathbb{R})$ and (\ref{eq-KK-56}) that
\beqq
\|f\|_{p}&\leq&\|f_{1}\|_{p}+\sum_{j=1}^{n}|f_{j}(0)|-|f_{1}(0)|+C_8C_9K(n-1)\|f_{1}\|_{p}\\
&\leq&\|f_{1}\|_{p}+\sqrt{n}|f(0)|-|f_{1}(0)|+C_8C_9K(n-1)\|f_{1}\|_{p}\\
&=&\left[C_8C_9K(n-1)+1\right]\|f_{1}\|_{p}-|f_{1}(0)|+\sqrt{n}|f(0)|,
\eeqq
which is what we need.

The proof of the sharpness of $p_{0}=1$ follows the same argument as in Theorem \ref{lem-im-1}.
\qed

\section{Riesz conjugate functions theorem for $\kappa$-pluriharmonic mappings}\label{csw-sec2}

The purpose of this section is to prove Proposition \ref{Prop-1} and Theorem \ref{thm-1}. Let us start the section with the following concept.

For a function $f\in \mathscr{H}(\mathbb{B}^{n}, \mathbb{C}^{n})$ $($i.e., $f$ is a holomorphic function of $\mathbb{B}^{n}$ into $\mathbb{C}^{n}$$)$, if
$$\|Df(z)\|\leq K|\det Df(z)|^{\frac{1}{n}},$$ then it is called a {\it Wu $K$-mapping} (see \cite{C-G,Wu}).

The following result concerning the relationship between holomorphic quasiregular mappings and Wu $K$-mappings is useful for the proof of Proposition \ref{Prop-1}.

\begin{ThmM}{\rm (cf. \cite[p.1373 and p.1385]{C-G})} \label{Wu-CG}
Every holomorphic $K$-quasiregular mapping is a Wu $K^{1-\frac{1}{n}}$-mapping, and
every Wu $K$-mapping is a $K^{n}$-quasiregular mapping.
\end{ThmM}

\subsection*{Proof of Proposition \ref{Prop-1}}

To prove the necessity part of the statement (1), assume that $f=h+\overline{g}\in \mathscr{PH}_{1}(\kappa)$. Then we have $$|g'(z)(h'(z))^{-1}|\leq\kappa$$ in $\mathbb{D}$.
This gives $$\frac{|h'(z)|+|g'(z)|}{|h'(z)|-|g'(z)|}\leq K,$$
which implies that
$$\left(|Df(z)|+|\overline{D}f(z)|\right)^{2}=(|h'(z)|+|g'(z)|)^{2}\leq K(|h'(z)|^{2}-|g'(z)|^{2})= KJ_{f}(z),$$
where $K=(1+\kappa)/(1-\kappa)$. Hence $f$ is a planar harmonic $K$-quasiregular mapping.

To show the sufficiency  part of the statement (1), assume that $f$ is a planar harmonic $K$-quasiregular mapping. Then it follows from \eqref{KQR} that
$$\left(|h'|+|g'|\right)^{2}\leq KJ_{f}=|h'|^{2}-|g'|^{2},$$
and thus, we get
$$\sup_{z\in\mathbb{D}}\left\{\frac{|g'(z)|}{|h'(z)|}\right\}\leq\kappa=(K-1)/(K+1).
$$
This demonstrates that $f\in\mathscr{PH}_{1}(\kappa)$.

For the proof of the necessity part of the statement (2), assume that $$f=h+\overline{g}\in \mathscr{PH}_{n}(\kappa)$$ is $K$-quasiregular with $K\geq 1$. Let
$$
\begin{cases}
f=(f_{1},\ldots,f_{n});\\
z=(z_{1},\ldots,z_{n});\\
f_{j}=u_{j}+iv_{j};\\
z_{j}=x_{j}+iy_{j}.
\end{cases}
$$
Then for $z\in\mathbb{B}^{n}$, we have
\be\label{gy-1}\Lambda_{f}(z):=\max_{\theta\in\mathbf{S}^{2n-1}}\{|J_{f}(z)\theta|\}\leq\, K|\det J_{f}(z)|^{\frac{1}{2n}},\ee
where \begin{equation*}
J_{f}=
\begin{bmatrix}
\frac{\partial u_{1}}{\partial x_{1}}&\frac{\partial
u_{1}}{\partial y_{1}}&\frac{\partial u_{1}}{\partial x_{2}}&\frac{\partial u_{1}}{\partial y_{2}}&\cdots&\frac{\partial u_{1}}{\partial x_{n}}&\frac{\partial u_{1}}{\partial y_{n}}\\
\frac{\partial v_{1}}{\partial x_{1}}&\frac{\partial
v_{1}}{\partial y_{1}}&\frac{\partial v_{1}}{\partial x_{2}}&\frac{\partial v_{1}}{\partial y_{2}}&\cdots&\frac{\partial v_{1}}{\partial x_{n}}&\frac{\partial v_{1}}{\partial y_{n}}\\
\vdots&\vdots&\vdots&\vdots&\ddots&\vdots&\vdots\\
\frac{\partial u_{n}}{\partial x_{1}}&\frac{\partial
u_{n}}{\partial y_{1}}&\frac{\partial u_{n}}{\partial x_{2}}&\frac{\partial u_{n}}{\partial y_{2}}&\cdots&\frac{\partial u_{n}}{\partial x_{n}}&\frac{\partial u_{n}}{\partial y_{n}}\\
\frac{\partial v_{n}}{\partial x_{1}}&\frac{\partial
v_{n}}{\partial y_{1}}&\frac{\partial v_{n}}{\partial x_{2}}&\frac{\partial v_{n}}{\partial y_{2}}&\cdots&\frac{\partial v_{n}}{\partial x_{n}}&\frac{\partial v_{n}}{\partial y_{n}}
\end{bmatrix}.
\end{equation*}
By calculations, we obtain

\beqq
|\det J_{f}|&=&\left|\det\begin{bmatrix}
Dh&\overline{Dg}\\
Dg&\overline{Dh}
\end{bmatrix}\right|
=
|\det Dh|^{2}\left|\det\left(I_{n}-\omega_{f}\overline{\omega_{f}}\right)\right|\\
&\leq&|\det Dh|^{2}\left\|I_{n}-\omega_{f}\overline{\omega_{f}}\right\|^{n}\\
&\leq&|\det Dh|^{2}\left(\|I_{n}\|+\left\|\omega_{f}\right\|^{2}\right)^{n}\\
&\leq&|\det Dh|^{2}(1+\kappa^{2})^{n},
\eeqq
which guarantees
\be\label{gy-2}
|\det J_{f}|^{\frac{1}{2n}}\leq\sqrt{1+\kappa^{2}}|\det Dh|^{\frac{1}{n}}.
\ee

Since
\beqq
\Lambda_{f}\geq\|Dh\|\left(1-\left\|\omega_{f}\right\|\right)\geq\|Dh\|(1-\kappa),
\eeqq
we deduce from (\ref{gy-1}) and (\ref{gy-2}) that
\beqq
\|Dh\|(1-\kappa)\leq\, K\sqrt{1+\kappa^{2}}|\det Dh|^{\frac{1}{n}},
\eeqq
whence
$$\|Dh\|\leq\frac{K\sqrt{1+\kappa^{2}}}{1-\kappa}|\det Dh|^{\frac{1}{n}}.$$
Then Theorem M ensures that $h$ is $K^{\ast}$-quasiregular, where $$K^{\ast}=\left(\frac{K\sqrt{1+\kappa^{2}}}{1-\kappa}\right)^{n}.$$

To check the sufficiency part of the statement (2), assume that $h$ is $K^*$-quasiregular with $K^*\geq 1$. Then it follows from Theorem M that
$h$ is a Wu $(K^*)^{1-1/n}$-mapping, which implies
\be\label{gy-3}\|Dh(z)\|\leq\, (K^*)^{1-\frac{1}{n}}|\det Dh(z)|^{\frac{1}{n}}\ \ \forall\ \ z\in\mathbb{B}^{n}.
\ee

Since
\beqq
|\det J_{f}(z)|&=&|\det\, Dh(z)|^{2}\left|\det\left(I_{n}-\omega_{f}(z)\overline{\omega_{f}(z)}\right)\right|\\
&\geq&|\det Dh(z)|^{2}\min_{\theta\in\partial\mathbb{B}^{n}}
\left\{\left|\left(I_{n}-\omega_{f}(z)\overline{\omega_{f}(z)}\right)\theta\right|^{n}\right\}\\
&\geq&|\det Dh(z)|^{2}\left(1-\|\omega_{f}(z)\|^{2}\right)^{n}\\
&\geq&|\det Dh(z)|^{2}(1-\kappa^{2})^{n},
\eeqq
we see that
\be\label{gy-55}
|\det Dh(z)|^{\frac{1}{n}}\leq\frac{1}{\sqrt{1-\kappa^{2}}}|\det J_{f}(z)|^{\frac{1}{2n}}.
\ee

Moreover, we infer from (\ref{gy-3}) that
\beqq
\Lambda_{f}(z)&=&\max_{\theta\in\mathbf{S}^{2n-1}}\{|J_{f}(z)\theta|\}=
\max_{\vartheta\in\partial\mathbb{B}^{n}}\left\{\left|Dh(z)\vartheta+\overline{Dg(z)}\vartheta\right|\right\}\\
&\leq&\|Dh(z)\|+\|Dh(z)[Dh(z)]^{-1}Dg(z)\|\\
&\leq&\|Dh(z)\|(1+\kappa)\\
&\leq& (K^*)^{1-\frac{1}{n}}(1+\kappa)|\det\, Dh(z)|^{\frac{1}{n}},
\eeqq
which, together with (\ref{gy-55}), ensures
$$\Lambda_{f}(z)\leq\,(K^*)^{1-\frac{1}{n}}\sqrt{\frac{1+\kappa}{1-\kappa}}|\det J_{f}(z)|^{\frac{1}{2n}}.$$
This illustrates that $f$ is a $K$-quasiregular mapping, where $$K=(K^*)^{1-\frac{1}{n}}\sqrt{\frac{1+\kappa}{1-\kappa}}.$$

Finally, we prove the statement (3) based on Proposition \ref{Prop-1}{\rm (2)}.
Let
$$
\begin{cases}
f=h+\kappa\overline{h}\ \ \forall\ \ \kappa\in[0,1);\\
h(z)=\left(z_{1},z_{2}+1/(1-z_{1})\right)\ \ \forall\ \  z=(z_{1},z_{2})\in\mathbb{B}^{2}.
\end{cases}
$$
It is not difficult to know that $\omega_f=\kappa$ and $h$ is biholomorphic. Hence $$f\in\mathscr{PH}_{2}(\kappa).$$
 However,
 $$
 \begin{cases} \sup_{z\in\mathbb{B}^{2}}\{\|Dh(z)\|\}=\infty;\\ J_{h}(z)=|\det Dh(z)|^{2}=1.
 \end{cases}
 $$
 These show that
$h$ is not $K$-quasiregular for any $K\geq 1$, and thus, it follows from Proposition \ref{Prop-1}{\rm (2)} that $f$ is not $K$-quasiregular for any $K\geq 1$.
\qed
\medskip

Before the proof of Theorem \ref{thm-1}, let us recall several useful results.
Let $\Omega\subset\mathbb{C}^{n}$ be a bounded symmetric domain containing the origin $0$ with Bergman-Silov boundary
$b$. Denote by $\Gamma$ the group of holomorphic automorphisms of $\Omega$, and by $\Gamma_{0}$ the
isotropy group of $\Gamma$. It is well known that $\Omega$ is circular and star-shaped with respect to
$0$ and that $b$ is circular. The group $\Gamma_{0}$ is transitive on $b$ and $b$ has a unique normalized
 $\Gamma_{0}$-invariant measure $\sigma$ with $\sigma(b)=1$ (see \cite{Bo,Hua,Lo}). Obviously, the unit polydisk and the
 unit ball $\mathbb{B}^{n}$ are bounded symmetric domains.
For convenience, we use the notation $\Omega_{0, b}$ to denote this class of bounded symmetric domains in $\mathbb{C}^{n}$.

\begin{ThmN}{\rm (\cite[Corollary 2.2]{CH-2023})}\label{Lem-A}
Let $g=u+iv\in \mathscr{H}(\Omega_{0, b}, \mathbb{C})$ $($i.e., $g$ is a holomorphic function of $\Omega_{0, b}$ into $\mathbb{C}$$)$ with $v(0)=0$. If
$$u\in \mathbf{h}^{p}(\Omega_{0, b},\mathbb{R})\ \ \text{ for some}\ \ p\in (1, \infty),
$$
then
$$
 \frac{\| v\|_p}{\cos(\pi/(2p^{\ast}))} \leq  \| g\|_p
\leq
\frac{\|u\|_p}{\sin(\pi/(2p^{\ast}))},
$$ where $p^{\ast}$ is defined in Theorem \ref{thm-1}.
\end{ThmN}

Let us recall the following Green formula which will be applied in the proof of Theorem \ref{thm-1}.

\begin{ThmO}{\rm (\cite[Lemma 4.1]{P-09})}\label{Green}
Suppose that $n\geq2$ and $\psi$ is a twice continuously differentiable function of $\mathbf{B}^{n}$. Then
\beqq\label{eq-1}
\int_{\mathbf{S}^{n-1}}\psi(r\zeta)\,d\sigma(\zeta)=
\psi(0)+\int_{\mathbf{B}^{n}_r}\Delta \psi(x)G_{n}(x,r)\,dV_{N}(x)\ \ \forall\ \ r\in (0,1),
\eeqq
 where
\beqq\label{Ge}
G_{n}(x,r)
=
\left\{
\begin{array}{ll}
\frac{|x|^{2-n}-r^{2-n}}{n(n-2)}, & n\geq3,
\\
\frac{1}{2}\log\frac{r}{|x|}, & n=2.
\end{array}
\right.
\eeqq
\end{ThmO}

For $p\in(0,\infty )$, denote by $L^{p}(\partial\mathbb{B}^{n})$ the
set of all measurable functions $F$ of $\partial\mathbb{B}^{n}$ into
$\mathbb{C}$ satisfying
$$\|F\|_{L^{p}}=\left(\int_{\partial\mathbb{B}^{n}}|F(\zeta)|^{p}d\sigma(\zeta)\right)^{\frac{1}{p}}<\infty.$$

Given $f\in\mathbf{H}^{p}(\mathbb{B}^{n},\mathbb{C})$, the Littlewood-Paley
type $\mathscr{G}$-function is defined as follows: For $\zeta\in \partial\mathbb{B}^{n}$,
$$\mathscr{G}(f)(\zeta)=\left(\int_{0}^{1}|D\,f(r\zeta)|^{2}(1-r)dr\right)^{\frac{1}{2}}.$$
It is known that for $p\in(1,\infty)$,
\beqq f\in\mathbf{H}^{p}(\mathbb{B}^{n},\mathbb{C})~\mbox{if and only
if}~ \mathscr{G}(f)\in\,L^{p}(\partial\mathbb{B}^{n}).\eeqq

For our application, we formulate the above result in the following form.

\begin{ThmP} {\rm (cf. \cite[Theorem 3.1]{AB})} \label{g-1-1}
Suppose that $p\in(1,\infty)$ and $f\in\mathscr{H}(\mathbb{B}^{n}, \mathbb{C})$. %Then there exists
\begin{enumerate}
\item[{\rm ($a_1$)}]\quad  If $f\in\mathbf{H}^{p}(\mathbb{B}^{n},\mathbb{C})$, then
$$
\exists\ \text{a positive constant}\ C_8=C_8(n,p)\ \text{ such that}\
\|\mathscr{G}(f)\|_{L^{p}}
\leq\,C_8\|f\|_{p}.
$$

\item[{\rm ($a_2$)}]\quad If $\mathscr{G}(f)\in\,L^{p}(\partial\mathbb{B}^{n})$, then $$
\exists\ \text{a positive constant}\ C_9=C_9(n,p)\ \text{such that}\
\|f\|_{p}\leq|f(0)|+C_9\|\mathscr{G}(f)\|_{L^{p}}.
$$
\end{enumerate}
\end{ThmP}

The following relation will be also exploited in the proof of Theorem \ref{thm-1}.

\begin{ThmQ}\label{Lem-H}{\rm (\cite[p.158]{kuang})}
For $n\geq 2$, let $a_{1},\ldots,a_{n}\in \mathbb{C}$.
Then
$$
\left(\sum_{k=1}^{n}|a_{k}|\right)^{p}
\leqslant\left\{\begin{array}{ll}
\sum_{k=1}^{n}|a_{k}|^{p} , & \text{if} \;\;\;0<p<1,\\
 n^{p-1}\sum_{k=1}^{n}|a_{k}|^{p}, & \text{if} \;\;\; p \geqslant 1.
 \end{array}\right.
 $$
\end{ThmQ}

\subsection*{Proof of Theorem \ref{thm-1}}
%To show the theorem, we only need to check the implication from $u\in \mathbf{h}^{p}(\mathbb{B}^{n},\mathbb{R}^{n})$
 %to $f\in\mathbf{h}^{p}(\mathbb{B}^{n},\mathbb{C}^{n})$ because the opposite implication from $f\in\mathbf{h}^{p}(\mathbb{B}^{n},\mathbb{C}^{n})$ to $u\in  %\mathbf{h}^{p}(\mathbb{B}^{n},\mathbb{R}^{n})$ is obvious.
 To prove the statement, assume that there is $p\in(1,\infty)$ such that $u\in \mathbf{h}^{p}(\mathbb{B}^{n},\mathbb{R}^{n})$. We consider two cases.

\noindent $\mathbf{Case~1.}$
Suppose that $p\in(1,2]$.

Let
$$
\begin{cases}
f=(f_{1},\ldots,f_{n});\\
u=(u_{1},\ldots,u_{n});\\
v=(v_{1},\ldots,v_{n}).
\end{cases}
$$
For $\mu\in\{1,2,\ldots\}$, define
$$
 F_{\mu}(z)=\left(|f(z)|^{2}+\frac{1}{\mu}\right)^{\frac{1}{2}}\ \ \&\ \
U_{\mu}(z)=\left(|u(z)|^{2}+\frac{1}{\mu}\right)^{\frac{1}{2}}\ \ \forall\ \ z\in\mathbb{B}^{n}.
$$

For the proof of this case, we need an upper bound for the quantity $M_{p}^{p}(r,F_{\mu})$ in terms of $M_{p}^{p}(r,U_{\mu})$, which is stated in \eqref{yao-1} below.
To reach this goal, let
$$
f_{z_{k}}=((f_{1})_{z_{k}},\ldots,(f_{n})_{z_{k}})\;\;\mbox\&\;\;f_{\overline{z}_{k}}=((f_{1})_{\overline{z}_{k}},\ldots,(f_{n})_{\overline{z}_{k}})\ \ \forall\ \  k\in\{1,2,\ldots,n\}.
$$

First, we obtain an upper bound for the quantity $\Delta(F_{\mu}^{p})$ in terms of $\|\overline{D}f\|$ as stated in \eqref{eq-4} below.
Since elementary computations give
$$\frac{\partial F_{\mu}^{p}}{\partial z_{k}}=\frac{p}{2}\left(|f|^{2}+\frac{1}{\mu}\right)^{\frac{p}{2}-1}\sum_{j=1}^{n}\Big[(f_{j})_{z_{k}}\overline{f_{j}}+f_{j}\overline{(f_{j})_{\overline{z}_{k}}}\Big],$$
it follows that
\beq\label{eq-2}\nonumber
\Delta(F_{\mu}^{p})&=&\frac{\sum_{k=1}^{n}
\left|\sum_{j=1}^{n}\Big[(f_{j})_{z_{k}}\overline{f_{j}}+f_{j}\overline{(f_{j})_{\overline{z}_{k}}}\Big]\right|^{2}}{\big(p(p-2)\big)^{-1}\left(|f|^{2}+\frac{1}{\mu}\right)^{2-\frac{p}{2}}}+\frac{\sum_{k=1}^{n}\left(|f_{z_{k}}|^{2}+|\overline{f_{\overline{z}_{k}}}|^{2}\right)} {(2p)^{-1}\left(|f|^{2}+\frac{1}{\mu}\right)^{1-\frac{p}{2}} }\\ \nonumber
&\leq&2p\left(|f|^{2}+\frac{1}{\mu}\right)^{\frac{p}{2}-1}\sum_{k=1}^{n}\left(|f_{z_{k}}|^{2}+|\overline{f_{\overline{z}_{k}}}|^{2}\right)\\
&=&2p\left(|f|^{2}+\frac{1}{\mu}\right)^{\frac{p}{2}-1}\left(\|Df\|_{F}^{2}+\|\overline{D}f\|_{F}^{2}\right).
\eeq

As the assumption of $f\in\mathscr{PH}_{n}(\kappa)$ implies that %$\left\|\overline{D}f [Df]^{-1}\right\|\leq \kappa$, we see that
\beq\label{eq-3}
\left\|\overline{\overline{D}f}\right\|=\left\|\overline{\overline{D}f}[Df]^{-1}Df\right\|\leq \left\|\overline{\overline{D}f} [Df]^{-1}\right\|\|Df\|\leq\kappa\|Df\|,
\eeq
and since (\ref{A-Frobenius-Operator}) gives
\beqq\label{hjp-01}\|Df\|_{F}^{2}+\|\overline{D}f\|_{F}^{2}\leq\,n\left(\|Df\|^{2}+\|\overline{D}f\|^{2}\right),\eeqq
we infer from  (\ref{eq-2}) that
\be\label{eq-4}
\Delta(F_{\mu}^{p})
\leq2pn(1+\kappa^{2})\left(|f|^{2}+\frac{1}{\mu}\right)^{\frac{p}{2}-1}\|Df\|^{2}.
\ee

Second, we establish a lower bound for the quantity $\Delta(U_{\mu}^{p})$ in terms of $\|\overline{D}f\|$ as stated in \eqref{eq-5} below.

By the similar arguments as in the proof of \eqref{eq-4}, we have
\beqq\label{eq-2.0}
\Delta(U_{\mu}^{p})&=&4p(p-2)\left(|u|^{2}+\frac{1}{\mu}\right)^{\frac{p}{2}-2}\sum_{k=1}^{n}\left|\sum_{j=1}^{n}u_{j}(u_{j})_{z_{k}}\right|^{2}\\
&&+
\frac{\sum_{j=1}^{n}|\nabla u_{j}|^{2}}{p^{-1} \left(|u|^{2}+\frac{1}{\mu}\right)^{1-\frac{p}{2}}},
\eeqq
and by Cauchy-Schwarz's inequality, we obtain
\beqq\label{eq-2fg.0}
\sum_{k=1}^{n}\left|\sum_{j=1}^{n}u_{j}(u_{j})_{z_{k}}\right|^{2}&\leq&|u|^{2}\sum_{k=1}^{n}|u_{z_{k}}|^{2}
=\frac{1}{4}|u|^{2}\sum_{j=1}^{n}|\nabla u_{j}|^{2}\\ \nonumber
&\leq&\frac{1}{4}\left(|u|^{2}+\frac{1}{\mu}\right)\sum_{j=1}^{n}|\nabla u_{j}|^{2},
\eeqq where $u_{z_{k}}=((u_{1})_{z_{k}},\ldots,(u_{n})_{z_{k}})$.
Then it follows from the assumption of $p\in(1,2]$ (as assumed in this case) that
\be\label{eq-cw-2}
\Delta(U_{\mu}^{p})\geq p(p-1)\left(|u|^{2}+\frac{1}{\mu}\right)^{\frac{p}{2}-1}\sum_{j=1}^{n}|\nabla u_{j}|^{2}.
\ee

Since for $j,k\in\{1,\ldots,n\}$, $$\left|(f_{j})_{z_{k}}+\overline{(f_{j})_{\overline{z}_{k}}}\right|^{2}=\big((u_{j})_{x_{k}}\big)^{2}+\big((u_{j})_{y_{k}}\big)^{2},$$ where $z_{k}=x_{k}+iy_{k}$,
 by (\ref{A-Frobenius-Operator}), we see that
\beqq\label{eq-cw-1}
\sum_{j=1}^{n}|\nabla u_{j}|^{2}=\left\|Df+\overline{\overline{D}f}\right\|_{F}^{2}\geq\left\|Df+\overline{\overline{D}f}\right\|^{2}\geq
\left(\|Df\|-\left\|\overline{\overline{D}f}\right\|\right)^{2},
\eeqq
and thus, we know from (\ref{eq-3}) and (\ref{eq-cw-2}) that
\beq\label{eq-5}
\Delta(U_{\mu}^{p})&\geq& p(p-1)\left(|u|^{2}+\frac{1}{\mu}\right)^{\frac{p}{2}-1}\left(\|Df\|-\left\|\overline{\overline{D}f}\right\|\right)^{2}\\ \nonumber
&\geq&p(p-1)(1-\kappa)^{2}\left(|u|^{2}+\frac{1}{\mu}\right)^{\frac{p}{2}-1}\|Df\|^{2}.
\eeq

Finally, since the assumption of $p\in(1,2]$ ensures that
 \beqq\label{eq-5-167}\left(|f(z)|^{2}+\frac{1}{\mu}\right)^{\frac{p}{2}-1}\leq\left(u^{2}(z)+\frac{1}{\mu}\right)^{\frac{p}{2}-1},\eeqq
and since Theorem O   and (\ref{eq-4}) lead to
\beqq
M_{p}^{p}(r,F_{\mu})&=&(F_{\mu}(0))^{p}+\int_{\mathbb{B}^{n}(0,r)}\Delta(|F_{\mu}(z)|^{p})G_{2n}(z,r)\,dV_{N}(z)\\ \nonumber
&\leq&(F_{\mu}(0))^{p}
+\frac{2np}{(1+\kappa^{2})^{-1}}\int_{\mathbb{B}^{n}(0,r)}\frac{|Df(z)|^{2}G_{2n}(z,r)}{\left(|f(z)|^{2}+{\mu}^{-1}\right)^{1-\frac{p}{2}}}\,dV_{N}(z),
\eeqq
we obtain from (\ref{eq-5}) that
\beq\label{yao-1}\nonumber
M_{p}^{p}(r,F_{\mu})&\leq& (F_{\mu}(0))^{p}\\ \nonumber
&&+\frac{2np}{(1+\kappa^{2})^{-1}}\int_{\mathbb{B}^{n}(0,r)}\frac{|Df(z)|^{2}G_{2n}(z,r)}{\left(u^{2}(z)+{\mu}^{-1}\right)^{1-\frac{p}{2}}}\,dV_{N}(z)\\ \nonumber
&\leq&(F_{\mu}(0))^{p}\\ \nonumber
&&+\frac{2n(1+\kappa^{2})}{(p-1)(1-\kappa)^{2}}\int_{\mathbb{B}^{n}(0,r)}\Delta(U_{\mu}^{p}(z))G_{2n}(z,r)\,dV_{N}(z)\\
&\leq&(F_{\mu}(0))^{p}+\frac{2n(1+\kappa^{2})}{(p-1)(1-\kappa)^{2}}\left(M_{p}^{p}(r,U_{\mu})-(U_{\mu}(0))^{p}\right),
\eeq
which is what we need.

As
$$M_{p}^{p}(r,v)\leq
M_{p}^{p}(r,f)=\lim_{\mu\rightarrow\infty}M_{p}^{p}(r,F_{\mu}),$$ we infer from \eqref{yao-1}, together with
the fact $f(0)=u(0)$ and Lebesgue's control-convergent theorem, that
\beq\label{eq-6}
M_{p}^{p}(r,v)&\leq&
\lim_{\mu\rightarrow\infty}(F_{\mu}(0))^{p}+\frac{2n(1+\kappa^{2})}{(p-1)(1-\kappa)^{2}}\lim_{\mu\rightarrow\infty}\left(M_{p}^{p}(r,U_{\mu})-(U_{\mu}(0))^{p}\right)\\
\nonumber
&\leq&|u(0)|^{p}+\frac{2n(1+\kappa^{2})}{(p-1)(1-\kappa)^{2}}\left(M_{p}^{p}(r,u)-|u(0)|^{p}\right)\\ \nonumber
&\leq&\frac{2n(1+\kappa^{2})}{(p-1)(1-\kappa)^{2}}M_{p}^{p}(r,u),
\eeq which implies the implication from $u\in \mathbf{h}^{p}(\mathbb{B}^{n},\mathbb{C}^{n})$ to $v\in \mathbf{h}^{p}(\mathbb{B}^{n},\mathbb{C}^{n})$. This ensures that $f\in\mathbf{h}^{p}(\mathbb{B}^{n},\mathbb{C}^{n})$.

Next, we come to prove the statement (a) for $\kappa=0$.
Since $\mathbb{B}^n$ is a bounded symmetric domain, it follows from Theorem N that for each $j\in\{1,\ldots,n\}$,
\be\label{eq-cw-9}M_{p}(r,v_{j})\leq~\cot(\pi/(2p^{\ast}))M_{p}(r,u_{j}).\ee
Then the Minkowski inequality gives
\beqq
M_{p}(r,v)&\leq&\left(\sum_{j=1}^{n}M_{p}^{2}(r,v_{j})\right)^{\frac{1}{2}}\\
&\leq&\cot(\pi/(2p^{\ast}))\left(\sum_{j=1}^{n}M_{p}^{2}(r,u_{j})\right)^{\frac{1}{2}}\\
&\leq&\cot(\pi/(2p^{\ast}))\left(\sum_{j=1}^{n}M_{p}^{2}(r,u)\right)^{\frac{1}{2}}\\
&\leq&\sqrt{n}\cot(\pi/(2p^{\ast}))M_{p}(r,u),
\eeqq as required.

\noindent $\mathbf{Case~2.}$
Suppose that $p\in (2,\infty)$.

The assumption of $f\in\mathscr{PH}_{n}(\kappa)$ implies that $f$ admits the canonical decomposition
$f = h + \overline{g}$, where $h=(h_{1},\ldots,h_{n})$ is locally biholomorphic and $g=(g_{1},\ldots,g_{n})$ is holomorphic in $\mathbb{B}^{n}$ with $g(0)=0$.
To prove that $f\in\mathbf{h}^{p}(\mathbb{B}^{n},\mathbb{R}^{n})$, clearly, it suffices to show that both $h$ and $g$ belong to $\mathbf{h}^{p}(\mathbb{B}^{n},\mathbb{C}^{n})$.

We first prove that $h\in\mathbf{h}^{p}(\mathbb{B}^{n},\mathbb{C}^{n})$. For this, let $$F=g+h\;\;\mbox{and}\;\;{\rm Im}(F)=(\widetilde{v}_{1},\ldots,\widetilde{v}_{n}).$$

Since ${\rm Re}(F)=u$, we see that
$${\rm Re}(F)\in\mathbf{h}^{p}(\mathbb{B}^{n},\mathbb{R}^{n}),$$
and thus, for each $j\in\{1,\ldots,n\}$, $u_{j}\in\mathbf{h}^{p}(\mathbb{B}^{n},\mathbb{R})$.

To continue the proof, let $$F_{j}=h_{j}+g_{j}$$ for each $j\in\{1,\ldots,n\}$.
In the following, we prove that $h_{j}\in\mathbf{h}^{p}(\mathbb{B}^{n},\mathbb{C}^{n})$ for each $j\in\{1,\ldots,n\}$.

Since $\widetilde{v}_{j}(0)=0$, by replacing $\Omega_{0,b}$ and $g$ by $\mathbb{B}^n$ and $F_j$, respectively, we obtain from Theorem N that
\beqq\label{eq-cw-5}
\frac{\|\widetilde{v}_{j}\|_{p}}{\cos(\pi/(2p))}\leq\|F_{j}\|_{p}\leq\frac{\|u_{j}\|_{p}}{\sin(\pi/(2p))}\leq\frac{\|u\|_p}{\sin(\pi/(2p))}.
\eeqq Here the fact of $p^*=p$ is applied.
Then it follows from
 Theorem P($a_1$) that there is a positive constant $C_6=C_6(n,p)$ such that
\beqq
\sum_{j=1}^{n}\int_{\partial\mathbb{B}^{n}}\big(\mathscr{G}(F_{j})(\zeta)\big)^{p}d\sigma(\zeta)&\leq&C_6\sum_{j=1}^{n}\|F_{j}\|_{p}^{p}\\
&\leq&\left(\frac{C_6}{\left(\sin(\pi/(2p))\right)^{p}}\right)\sum_{j=1}^{n}\|u_{j}\|_{p}^{p}\\
&\leq&
\left(\frac{n C_6}{\left(\sin(\pi/(2p))\right)^{p}}\right)\|u\|_{p}^{p}.
\eeqq
Hence we get from Theorem Q and the following fact
$$\|DF\|_{F}^{2}=\sum_{k=1}^{n}\sum_{j=1}^{n}\left|(F_{j})_{z_{k}}\right|^{2}=\sum_{j=1}^{n}|DF_{j}|^{2}$$
that
\beq\label{eq-cw-7} \quad\quad
I_{F}&\leq&n^{\frac{p}{2}-1}
\sum_{j=1}^{n}\int_{\partial\mathbb{B}^{n}}\big(\mathscr{G}(F_{j})(\zeta)\big)^{p}d\sigma(\zeta)\\ \nonumber
&\leq&\left(\frac{n^{\frac{p}{2}}C_6}{\left(\sin(\pi/(2p))\right)^{p}}\right)\|u\|_{p}^{p},
\eeq where $$I_{F}=\int_{\partial\mathbb{B}^{n}}\left(\int_{0}^{1}\|DF(r\zeta)\|_{F}^{2}(1-r)dr\right)^{\frac{p}{2}}d\sigma(\zeta).$$

Moreover, since
\be\label{vv-01}\|Dg\|=\|Dg[Dh]^{-1}Dh\|\leq\|\omega_{f}\|\|Dh\|\leq\kappa\|Dh\|,\ee
we deduce from (\ref{A-Frobenius-Operator})
that
\beqq \|DF\|_{F}^{2}&\geq&\|DF\|^{2}\geq(\|Dh\|-\|Dg\|)^{2}\geq(1-\kappa)^{2}\|Dh\|^{2}\\
&\geq&{(1-\kappa)^{2}}{n}^{-1}\|Dh\|_{F}^{2},\eeqq
and so, by (\ref{eq-cw-7}),
we get
\beq\label{eq-cw-8}\nonumber
\int_{\partial\mathbb{B}^{n}}\left(\int_{0}^{1}\|Dh(r\zeta)\|_{F}^{2}(1-r)dr\right)^{\frac{p}{2}}d\sigma(\zeta)
&\leq&\frac{n^{\frac{p}{2}}I_F}{(1-\kappa)^{p}}\\
&\leq&\frac{n^{p}C_6\|u\|_{p}^{p}}{(1-\kappa)^{p}\left(\sin(\pi/(2p))\right)^{p}}.
\eeq
This gives
\beqq
\sum_{j=1}^{n}\int_{\partial\mathbb{B}^{n}}\big(\mathscr{G}(h_{j})(\zeta)\big)^{p}d\sigma(\zeta)&\leq& n\int_{\partial\mathbb{B}^{n}}\left(\int_{0}^{1}\|Dh(r\zeta)\|_{F}^{2}(1-r)dr\right)^{\frac{p}{2}}d\sigma(\zeta)\\
&\leq&\frac{n^{p+1}C_6\|u\|_{p}^{p}}{(1-\kappa)^{p}\left(\sin(\pi/(2p))\right)^{p}},
\eeqq
and thus, Theorem P($a_2$) guarantees that  $$h_{j}\in\mathbf{h}^{p}(\mathbb{B}^{n},\mathbb{C}^{n})\ \ \forall\ \ j\in\{1,\ldots,n\},
$$ as required.

As Theorem Q leads to
\beqq
M_{p}(r,h)&=&\left(\int_{\partial\mathbb{B}^{n}}\left(\sum_{j=1}^{n}|h_{j}(r\zeta)|^{2}\right)^{\frac{p}{2}}d\sigma(\zeta)\right)^{\frac{1}{p}}\\
&\leq&n^{\frac{p-2}{2p}}\left(\sum_{j=1}^{n}\int_{\partial\mathbb{B}^{n}}\sum_{j=1}^{n}|h_{j}(r\zeta)|^{p}d\sigma(\zeta)\right)^{\frac{1}{p}}\\
&\leq&n^{\frac{p-2}{2p}}\left(\sum_{j=1}^{n}M_{p}(r,h_{j})\right),
\eeqq
and since $$
h_{j}\in\mathbf{h}^{p}(\mathbb{B}^{n},\mathbb{C}^{n})\ \ \forall\ \ j\in\{1,\ldots,n\},
$$
we know
$$
h\in\mathbf{h}^{p}(\mathbb{B}^{n},\mathbb{C}^{n}).
$$

Next, we show that $g\in\mathbf{h}^{p}(\mathbb{B}^{n},\mathbb{C}^{n}).$ To reach this goal, let $$I_{g}=\int_{\partial\mathbb{B}^{n}}\left(\int_{0}^{1}\|Dg(r\zeta)\|_{F}^{2}(1-r)dr\right)^{\frac{p}{2}}d\sigma(\zeta).$$
Then (\ref{A-Frobenius-Operator}) gives
\beqq
I_{g}&\leq&n^{\frac{p}{2}}
\int_{\partial\mathbb{B}^{n}}\left(\int_{0}^{1}\|Dg(r\zeta)\|^{2}(1-r)dr\right)^{\frac{p}{2}}d\sigma(\zeta).
\eeqq

Due to $$h\in\mathbf{h}^{p}(\mathbb{B}^{n},\mathbb{C}^{n}),$$ it follows from (\ref{A-Frobenius-Operator}),
 (\ref{vv-01}) and (\ref{eq-cw-8}) that
\beq\label{arb-1}
I_{g}
&\leq&n^{\frac{p}{2}}\kappa^{p}
\int_{\partial\mathbb{B}^{n}}\left(\int_{0}^{1}\|Dh(r\zeta)\|^{2}(1-r)dr\right)^{\frac{p}{2}}d\sigma(\zeta)\\ \nonumber
&\leq&n^{\frac{p}{2}}\kappa^{p}
\int_{\partial\mathbb{B}^{n}}\left(\int_{0}^{1}\|Dh(r\zeta)\|_{F}^{2}(1-r)dr\right)^{\frac{p}{2}}d\sigma(\zeta)\\ \nonumber
&\leq&\frac{\kappa^{p}n^{\frac{3p}{2}}C_6\|u\|_{p}^{p}}{(1-\kappa)^{p}\left(\sin(\pi/(2p))\right)^{p}}.
\eeq
This shows that
\beqq
\sum_{j=1}^{n}\int_{\partial\mathbb{B}^{n}}\big(\mathscr{G}(g_{j})(\zeta)\big)^{p}d\sigma(\zeta)\leq nI_{g}\leq\frac{\kappa^{p}n^{\left(\frac{3p}{2}+1\right)}C_6\|u\|_{p}^{p}}{(1-\kappa)^{p}\left(\sin(\pi/(2p))\right)^{p}},
\eeqq
and thus, Theorem P($a_2$) yields that
\beqq\label{knv-1}g_{j}\in\mathbf{h}^{p}(\mathbb{B}^{n},\mathbb{C}^{n})\ \ \forall\ \ j\in\{1,\ldots,n\} .\eeqq

Since the Minkowski inequality gives
\beqq
M_{p}(r,g)\leq\left(\sum_{j=1}^{n}M_{p}^{2}(r,g_{j})\right)^{\frac{1}{2}}\leq\sum_{j=1}^{n}M_{p}(r,g_{j}),
\eeqq we get
 $$g\in\mathbf{h}^{p}(\mathbb{B}^{n},\mathbb{C}^{n}).$$

Now, we  prove the statement (b) for $\kappa=0$.
By Theorem N, we know
\be\label{eq-cw-71}M_{p}(r,v_{j})\leq~\cot(\pi/(2p))M_{p}(r,u_{j})\ \ \forall\ \ j\in\{1,\ldots,n\},\ee
and thus, the Minkowski inequality gives
\beqq
M_{p}(r,v)&\leq&\left(\sum_{j=1}^{n}M_{p}^{2}(r,v_{j})\right)^{\frac{1}{2}}\leq\cot(\pi/(2p))\left(\sum_{j=1}^{n}M_{p}^{2}(r,u_{j})\right)^{\frac{1}{2}}\\
&\leq&\cot(\pi/(2p))\left(\sum_{j=1}^{n}M_{p}^{2}(r,u)\right)^{\frac{1}{2}}\\
&\leq&\sqrt{n}\cot(\pi/(2p))M_{p}(r,u),
\eeqq which is what we need.

The proof that $p_{0}=1$ is sharp is the same as that in Theorem \ref{lem-im-1}.
\qed

\bigskip

{\bf Acknowledgements:} %We would like to thank  Professors David Kalaj and Jinsong Liu for their
%useful suggestions to this paper.
S. L. Chen was partially supported by the
National Science Foundation of China (No. 1257011497), the
Hunan Provincial Natural Science Foundation of China (No. 2022JJ10001), the
Key Projects of Hunan Provincial Department of Education (No. 21A0429),
the Double First-Class University Project of Hunan Province (Xiangjiaotong [2018]469),
the Science and Technology Plan Project of Hunan Province (No. 2016TP1020),
and the Discipline Special Research Projects of Hengyang Normal University (No. XKZX21002). M. Z. Huang \& X. T. Wang were partially supported by NNSFs of China (No. 12371071). J. Xiao was supported by NSERC of Canada (No. 202979).

%\section*{Statements and Declarations}

%\subsection*{Competing interests}
%There are no competing interests.

%\subsection*{Data availability}
%Data sharing not applicable to this article as no datasets were generated or analysed during the current study.

%\subsection{subsection 3.1}

%\section{Appendix}

\normalsize

\end{document}